\documentclass{amsart}
\usepackage[T1]{fontenc}
\usepackage{lmodern}
\usepackage{cite,amssymb,microtype,a4wide}
\usepackage{mathtools,comment}
\usepackage{xcolor}
\usepackage{tikz}
\usepackage[colorlinks=true,linkcolor=blue,citecolor=blue,urlcolor=black]{hyperref}
\newtheorem{theorem}{Theorem}[section]
\newtheorem{lemma}[theorem]{Lemma}
\newtheorem{proposition}[theorem]{Proposition}
\newtheorem{corollary}[theorem]{Corollary}
\theoremstyle{remark}
\newtheorem{remark}[theorem]{Remark}
\numberwithin{equation}{subsection}
\newcommand{\Tr}{\operatorname{Tr}}
\newcommand{\Rea}{\operatorname{Re}}
\newcommand{\Ima}{\operatorname{Im}}
\newcommand{\row}{\mathrm{row}}
\newcommand{\col}{\mathrm{col}}
\DeclareRobustCommand{\newtext}[1]{#1}
\DeclareRobustCommand{\revision}[1]{#1}
\newtheorem*{remark*}{Remark}

\title[Type III and semifinite noncommutative $L_p$-spaces]
{$L_p$-spaces, $0<p\ne 2<\infty$, associated with
type III and semifinite von Neumann algebras \newtext{are not isomorphic}}
\author[J. Huang]{Jinghao Huang}
\address{Institute for Advanced Study in Mathematics of HIT,
Harbin Institute of Technology, Harbin, 150001, China}
\email{\color{blue}jinghao.huang@hit.edu.cn}

\author[F. Sukochev]{Fedor Sukochev}
\address{School of Mathematics and Statistics, University of NSW,
Sydney, 2052, Australia}
\email{\color{blue}f.sukochev@unsw.edu.au}

\thanks{J. Huang was supported by the NNSF of China
(No. 12301160, 12471134 and 12671159);
F. Sukochev was supported by the ARC (DP230100434).}
\subjclass[2020]{Primary 46L51; Secondary 46B20, 46L10, 46L54}
\keywords{Noncommutative $L_p$-space, type III algebra, conditional expectation,
isomorphic embedding, quasi-Banach space, equiintegrability}
\hypersetup{pdftitle={Lp-spaces (0 < p < infinity, p not equal to 2) associated with type III and semifinite von Neumann algebras are not isomorphic},
pdfauthor={Jinghao Huang and Fedor Sukochev}}
\begin{document}
\begin{abstract}
Let $M$ be a nonzero type III von Neumann algebra, and let $N$ be a semifinite von Neumann algebra.
We prove that the Banach space $M_*$ does not embed
isomorphically into the Banach space $N_*$. More generally, we prove that
$L_p(M)$ is not isomorphic to a subspace of
$L_p(N)$ for $0<p<2$.
In particular, the linear isomorphism class of $L_p(M)$
distinguishes type III algebras from semifinite algebras for
$0<p\ne 2<\infty$. This answers several long-standing open questions in the literature.
\end{abstract}
\maketitle

\section{Introduction}

The distinction between sequence spaces and function spaces is
a classical feature of the isomorphic theory of $L_p$-spaces\newtext{; see}
\cite[Ch.~IX, \S2, and Ch.~XII, \S3, Theorem~9]{Banach1932}
\newtext{and} \cite[p.~138]{AlbiacKalton2016}.
For $1\le p<\infty$, $p\ne2$, the function space $L_p(0,1)$
\newtext{does not embed isomorphically into the sequence space $\ell_p$.}
McCarthy~\cite{McCarthy} established a noncommutative
non-embedding theorem: for $1\le p\ne 2<\infty$, the Schatten
class $\mathcal{C}_p=\mathcal{C}_p(\ell_2)$ is not isomorphic to a subspace of any
commutative $L_p$-space.\footnote{Arazy and
Lindenstrauss~\cite[p.~84, (iii)]{AL75} note that details of
McCarthy's original argument require correction when $1\le p<2$.
For an alternative proof in this range, see Gordon and
Lewis~\cite[Theorem~5.6(e) and the subsequent remark]{GL74}.}
In the reverse direction, Arazy and Lindenstrauss~\cite[Theorem~6]{AL75} showed that $L_p(0,1)$
does not embed isomorphically into $\mathcal C_p$ for $1\le p<\infty$, $p\ne2$.

These classical results lead to the question of which properties
of a von Neumann algebra can be recovered from the Banach
isomorphism class of its $L_p$-space.
We consider the distinction between type III and semifinite
algebras. We use the term type III algebra only for nonzero algebras.
We consider complex Banach spaces and, when $0<p<1$, complex
quasi-Banach spaces. An isomorphism is a bounded linear bijection
with bounded inverse. For two such spaces $X$ and $Y$, we write $X\hookrightarrow Y$
if there exists a linear isomorphic embedding $T:X\to Y$, that is,
\[
 c\left\|x\right\|_X\le\left\|Tx\right\|_Y
 \le C\left\|x\right\|_X\qquad(x\in X)
\]
for some constants $0<c\le C<\infty$.
We also write $T:X\hookrightarrow Y$ when the map is specified,
and $X\not\hookrightarrow Y$ when no such map exists.
We write $\cong$ for a normal $*$-isomorphism between von Neumann
algebras and for the induced surjective linear isometry between their
$L_p$-spaces.
The spaces $L_p(M)$ are the Haagerup spaces; the predual of a
von Neumann algebra $M$ is identified with $L_1(M)$.
For the general theory of noncommutative $L_p$-spaces, see~\cite{PX}.

\newtext{This line of inquiry} was later extended to
noncommutative $L_p$-spaces over finite tracial von Neumann algebras.
The second author~\cite{Sukochev} proved that $\mathcal{C}_p$ does not embed into
such an $L_p$-space for
$2<p<\infty$. Haagerup, Rosenthal and the second author~\cite{HRS}
established the corresponding result for $1\le p<2$.
Since $L_p(M)$ contains a copy of $\mathcal{C}_p$ whenever $M$ is properly infinite,
these theorems exclude embeddings of $L_p(M)$ into
noncommutative $L_p$-spaces over a finite tracial von Neumann
algebra. The second author and Xu~\cite{SX} subsequently
obtained the analogous result for $0<p<1$.

Building on the study of the type I case~\cite{Sukochev2000},
Haagerup, Rosenthal and the second author~\cite{HRS} (see also \cite{SukochevChilin1988,SukochevChilin1990,Sukochev2001,HRS2000})
classified the separable infinite-dimensional $L_p$-spaces of
hyperfinite semifinite von Neumann algebras into thirteen
isomorphism classes, for $1\le p<\infty$, $p\ne2$.
They also proved that the $L_p$-spaces of the hyperfinite
type $\mathrm{III}_\lambda$ factors, $0<\lambda\le1$, are
completely isomorphic.

However, \newtext{the precise nature of the relationship between the}
$L_p$-spaces of \newtext{type $\mathrm{II}_\infty$ and type III von Neumann
algebras remained a major unresolved issue.} Haagerup et al. therefore asked~\cite{HRS2000,HRS}:
\begin{quote}
Our methods do not cover the following case, which remains a fascinating open problem: Is it so
that the predual of a type III von-Neumann algebra does not Banach embed in the predual of one
of type II$_\infty$? In fact, we do not know if the predual of the injective type II$_\infty$ factor can be Banach
isomorphic to the predual of an injective type III-factor.\cite[p. 5]{HRS} 
$\cdots$
It would be interesting to know, whether a type
$\mathrm{II}_\infty$-factor and a type III-factor can be
distinguished by the Banach space isomorphism classes of their
preduals.\cite[p. 61]{HRS}
\end{quote}
Pisier and Xu~\cite[Problem~10.9, p.~1509]{PX} \newtext{posed} the
corresponding question\newtext{, omitting} factoriality \newtext{and} injectivity
assumptions:
\begin{quote}
Let $M$ and $N$ be two von Neumann algebras of type
$\mathrm{II}_\infty$ and III, respectively. Are $L^p(M)$ and
$L^p(N)$ isomorphic for $p\ne2$?
\end{quote}

Several results distinguish the two types under additional
assumptions on the embedding. Oikhberg, Rosenthal and
St{\o}rmer~\cite[Problem~2]{ORS} asked 
\begin{quote}
whether a von Neumann
algebra $M$ must be semifinite if its predual embeds isomorphically
into the predual of a semifinite von Neumann algebra.
\end{quote}
They answered this question affirmatively for isometric
embeddings~\cite[Theorem~1]{ORS}.
It was shown in \cite{HPS2023} that the
predual of any atomless von Neumann algebra does not embed isomorphically
into any KB-ideal of operators on $\ell_2$, which gives
$M_*\not\hookrightarrow\mathcal C_1=B(\ell_2)_*$ for every type III
von Neumann algebra $M$.
Junge~\cite{Junge} embedded the operator Hilbert space $OH$
completely isomorphically into a von Neumann predual, whereas
Pisier~\cite{Pisier} excluded completely isomorphic embeddings
of $OH$ into semifinite preduals.
Pisier~\cite[Remarks~(iv), p.~459]{Pisier} also recalled the open
question \emph{whether the predual of an injective type
$\mathrm{III}_\lambda$ factor, $0<\lambda\le1$, embeds
isomorphically as a Banach space into the predual of a type
$\mathrm{II}_\infty$ von Neumann algebra}. Pisier described the
obstruction to completely isomorphic embeddings as a partial
answer to this Banach embedding question.
Haagerup and Musat~\cite[Theorem~1.1]{HM} proved that a hyperfinite
type III predual cannot be completely isomorphic to a completely
complemented subspace of a hyperfinite semifinite predual, when
both algebras have separable preduals. Haagerup and Musat also obtained
uncountably many complete-isomorphism classes of hyperfinite
type $\mathrm{III}_0$ preduals~\cite[Theorem~1.3]{HM}.

The following theorem and its duality corollary resolve \newtext{the} problems \newtext{stated above}.

\begin{theorem}\label{thm:unified}
Let $M$ be a nonzero type III von Neumann algebra, and let $N$ be a semifinite von Neumann algebra.
For $0<p<2$, we have $L_p(M)\not\hookrightarrow L_p(N)$.
\end{theorem}
{
The result above \newtext{distinguishes $L_p(R_\infty)$ from each of the thirteen
semifinite spaces in Figure~\ref{fig:hrs-typeiii}, giving fourteen pairwise
nonisomorphic $L_p$-spaces for $1\le p<2$.}

For $1\le p<2$, Figure~\ref{fig:hrs-typeiii} below extends the Hasse
diagram in \cite[Figure~1]{HRS} by adding $L_p(R_\infty)$, where
$R_\infty$ is the hyperfinite type $\mathrm{III}_1$ factor.
In the figure, we write
\[
 S_p:=\left(\bigoplus_{n=1}^{\infty}\mathcal{C}_p^n\right)_{\ell_p},
 \qquad \mathcal{C}_p^n:=\mathcal{C}_p(\mathbb C^n),\qquad
 L_p:=L_p(0,1).
\]
Let $R$ be the hyperfinite type $\mathrm{II}_1$ factor.
As in \cite{HRS}, $L_p\otimes_p X$ denotes the Bochner space
$L_p((0,1);X)$, and
$L_p(R)\otimes_p\mathcal{C}_p:=L_p(R\overline\otimes B(\ell_2))$.
Let $M$ be a nonzero hyperfinite semifinite von Neumann algebra and
let $N$ be any nonzero type III von Neumann algebra, both with
separable preduals. \textcolor{black}{We have} the following ($1$-complemented ) isometric embeddings
\[
 L_p(M)\hookrightarrow L_p(R\overline\otimes B(\ell_2))
 \hookrightarrow L_p(N)\qquad(1\le p<\infty).
\]
For the first embedding, the equality
\eqref{eq:semifinite-stabilization} below realizes $M$ as
$q(F\overline\otimes B(\ell_2))q$, where $F$ is a finite and
hyperfinite algebra and $q\in\operatorname{Proj}(F\overline\otimes B(\ell_2))$.
On the other hand, 
  $F$ can be viewed as a unital von Neumann subalgebra of $R$, as in \cite[p.~59, (5.48) and the following paragraph]{HRS}, and let $E_F:R\to F$ be the
trace-preserving conditional expectation. Then
\[
 E_0:R\overline\otimes B(\ell_2)\longrightarrow
 q(F\overline\otimes B(\ell_2))q,\qquad
 E_0(x):=q(E_F\overline\otimes\mathrm{id}_{B(\ell_2)})(x)q
\]
is a normal contractive projection. 
For the second embedding, \cite[Theorem~3.5 and its proof]{PX} gives a copy of
$R$ in $N$ which is the range of a normal conditional expectation.
Since $N$ is properly infinite and has separable predual,
$N\cong N\overline\otimes B(\ell_2)$.
This proves the $1$-complemented isometric embeddings. 
Thus, all thirteen semifinite spaces in Figure~\ref{fig:hrs-typeiii}
embed isometrically as $1$-complemented subspaces of $L_p(N)$; taking $N=R_\infty$ gives
the bottom in the figure.


Moreover, for every nonzero hyperfinite type III von Neumann
algebra $M$ with separable predual and every $1\le p<\infty$,
$L_p(M)$ is isometric to a $1$-complemented subspace of
$L_p(R_\infty)$. 
Indeed, if $M$ is a factor, then
  $M\overline\otimes R_\infty\cong R_\infty$~\cite[Lemma~3.8]{HM}.
For a faithful normal state $\omega$ on $R_\infty$, the map
\[
 J_p:L_p(M)\longrightarrow
 L_p(M\overline\otimes R_\infty)\cong L_p(R_\infty),\qquad
 J_p(x):=x\otimes\omega^{1/p}
\]
is isometric, since
$\left\|x\otimes\omega^{1/p}\right\|_p=\left\|x\right\|_p$.
The faithful normal conditional expectation
$\mathrm{id}_M\otimes\omega$ yields the corresponding contractive
projection on $L_p$; see \eqref{eq:expected-lp}.
For a general hyperfinite type $III$ von Neumann algebra $M$  (which is not necessarily  a factor), applying \cite[Lemma~3.8]{HM} to the factors in
its central decomposition and then using
\cite[Theorem~3, p.~437]{TakesakiReduction1969} gives
\[
 M\overline\otimes R_\infty
 \cong Z(M)\overline\otimes R_\infty.
\]
Consider
  $Z(M)$  as a unital
von Neumann subalgebra of a diffuse maximal abelian subalgebra of $R$.
The trace-preserving conditional expectation from $R$ onto this copy
of $Z(M)$ gives an inclusion with faithful normal conditional
expectation
\[
 Z(M)\overline\otimes R_\infty
 \subset R\overline\otimes R_\infty\cong R_\infty.
\]
This proves the assertion for the case $M$ is not a factor.

}
\begin{figure}[htbp]
\centering
\scalebox{0.8}{%
\begin{tikzpicture}[
 x=1.6cm,y=1cm,
 vertex/.style={circle,fill=black,inner sep=0pt,minimum size=4pt},
 every label/.style={font=\small,inner sep=3pt},
 line width=0.45pt]
 \node[vertex,label=above:{$\ell_p$}] (ell) at (0,0) {};
 \node[vertex,label=left:{$S_p$}] (blocks) at (-1,-1) {};
 \node[vertex,label=right:{$L_p$}] (comm) at (1,-1) {};
 \node[vertex,label=left:{$\mathcal{C}_p$}] (schatten) at (-2,-2) {};
 \node[vertex,label=right:{$S_p\oplus L_p$}]
   (blockscomm) at (0,-2) {};
 \node[vertex,label=left:{$\mathcal{C}_p\oplus L_p$}]
   (schattencomm) at (-1,-3) {};
 \node[vertex,label=right:{$L_p\otimes_p S_p$}]
   (commblocks) at (1,-3) {};
 \node[vertex,label=right:{$\mathcal{C}_p\oplus(L_p\otimes_p S_p)$}]
   (mixed) at (0,-4) {};
 \node[vertex,label=right:{$L_p(R)$}] (finite) at (2,-4) {};
 \node[vertex,label=left:{$L_p\otimes_p\mathcal{C}_p$}]
   (commschatten) at (-1,-5) {};
 \node[vertex,label=right:{$\mathcal{C}_p\oplus L_p(R)$}]
   (schattenfinite) at (1,-5) {};
 \node[vertex,label=right:{$L_p(R)\oplus(L_p\otimes_p\mathcal{C}_p)$}]
   (finitemixed) at (0,-6) {};
 \node[vertex,label=right:{$L_p(R)\otimes_p\mathcal{C}_p$}]
   (semifinite) at (0,-7) {};
 \node[vertex,label=right:{$L_p(R_\infty)\quad
       (\text{type }\mathrm{III}_1)$}]
   (typeiii) at (0,-8) {};
 \draw (ell)--(blocks) (ell)--(comm)
       (blocks)--(schatten) (blocks)--(blockscomm)
       (comm)--(blockscomm)
       (schatten)--(schattencomm) (blockscomm)--(schattencomm)
       (blockscomm)--(commblocks)
       (schattencomm)--(mixed) (commblocks)--(mixed)
       (commblocks)--(finite)
       (mixed)--(commschatten) (mixed)--(schattenfinite)
       (finite)--(schattenfinite)
       (commschatten)--(finitemixed) (schattenfinite)--(finitemixed)
       (finitemixed)--(semifinite) (semifinite)--(typeiii);
\end{tikzpicture}
}
\caption{
Among distinct displayed spaces,
$X\hookrightarrow Y$ if and only if a descending path joins $X$
to $Y$.}
\label{fig:hrs-typeiii}
\end{figure}

\begin{corollary}\label{cor:duality}
Let $M$ and $N$ be as in Theorem~\ref{thm:unified}.
For $2<p<\infty$, the space $L_p(M)$ is not isomorphic to a
complemented subspace of $L_p(N)$. Consequently, $L_p(M)$ and
$L_p(N)$ are not linearly isomorphic for
$0<p<\infty$, $p\ne2$.
\end{corollary}

Readers interested only in preduals of von Neumann algebras ($p=1$)
may skip Section~\ref{sec:array-proof}. It suffices to read
Sections~\ref{sec:prelim} and~\ref{sec:endpoint}, followed by
Lemma~\ref{lem:separable-support} and the $p=1$ part of the proof of
Theorem~\ref{thm:unified} in Subsection~\ref{sec:nonfactor}.

For embeddings of symmetric
operator spaces into noncommutative $L_p$-spaces over finite
tracial von Neumann algebras, see~\cite{HJSZ25,JSZ2017,Randrianantoanina2008};
for isometric embeddings of noncommutative $L_p$-spaces into
symmetric operator spaces, see~\cite{HJSZ,Junge2000}.

\section{Preliminaries}\label{sec:prelim}
\subsection{Notation and standard results}

All algebras and linear spaces in the present paper are complex. A bounded linear
map between von Neumann algebras is said to be  normal if it is ultraweakly
continuous.
We refer to Takesaki \cite{Takesaki} for von Neumann algebras
  and to \cite{PX,Hiai2021,DPS2023} for noncommutative integration.

For $n\ge1$ and a von Neumann algebra $A$, let $\mathbb{M}_n(A)$
denote the algebra of $n\times n$ matrices with entries in $A$, and set
$\mathbb{M}_n=\mathbb{M}_n(\mathbb C)$.
We write $\Tr_n$ for the unnormalized matrix trace on $\mathbb{M}_n$,
so that $\Tr_n(1_{\mathbb{M}_n})=n$.

For a semifinite von Neumann  algebra $(A,\tau)$ we identify $A_*$ with
$L_1(A,\tau)$ through the bilinear pairing
\[
                   \langle y,a\rangle=\tau(ya).
\]
{For a general von Neumann algebra $A$, choose a faithful
normal semifinite weight $\psi$ and form
$\mathcal B=A\rtimes_{\sigma^\psi}\mathbb R$, where $\sigma^\psi$
is the modular automorphism group of $\psi$.}
{Let $\widehat\sigma$ be the dual action
{on the crossed product $\mathcal B$},
$\tau_{\mathcal B}$ the canonical trace, and $L_0(\mathcal B,\tau_{\mathcal B})$
the space of $\tau_{\mathcal B}$-measurable operators affiliated with
$\mathcal B$.}
{The Haagerup spaces are defined by
\cite[Section~1.2, pp.~2129--2131]{HJX}}
\begin{equation}\label{eq:haagerup-space}
 {
 \begin{gathered}
 L_p(A):=\{x\in L_0(\mathcal B,\tau_{\mathcal B}):
       \widehat\sigma_t(x)=e^{-t/p}x\ \text{for all }t\in\mathbb R\}
       \quad(0<p<\infty),\\
 L_\infty(A):=A.
 \end{gathered}}
\end{equation}
{Under the canonical identification $A_*\cong L_1(A)$,
write $D_\omega\in L_1(A)_+$ for the density of $\omega\in A_*^+$.}
{The canonical positive functional
$\operatorname{tr}:L_1(A)\to\mathbb C$ and the $L_p$ quasi-norms satisfy}
\begin{equation}\label{eq:haagerup-norm}
 {
 \begin{gathered}
 \operatorname{tr}(D_\omega)=\omega(1)\qquad(\omega\in A_*^+),\\
 \left\|x\right\|_p^p=\operatorname{tr}(|x|^p)
 =\left\||x|^p\right\|_1
 \qquad(x\in L_p(A),\ 0<p<\infty).
 \end{gathered}}
\end{equation}
For $0<p<1$, we use the $p$-triangle inequality
$\|x+y\|_p^p\le \|x\|_p^p+\|y\|_p^p$ for $x,y\in L_p(A)$
\cite[Section~1.2]{HJX}.
{We use the notation $\omega^{1/p}:=D_\omega^{1/p}$ and
the pairing formula \cite[p.~2131, (1.13)]{HJX}}
\begin{equation}\label{eq:haagerup-pairing}
 {
 \operatorname{tr}(D_\omega a)=\omega(a)
 \qquad(\omega\in A_*^+,\ a\in A).}
\end{equation}
{The spaces $L_p(A)$ are independent of $\psi$ up to
canonical isometric identification \cite[Remark~1.2]{HJX}.}

For a projection $e\in A$, the compression
$C_e:A\to eAe$, $C_e(a)=eae$, is normal and contractive.
The preadjoint of $C_e$ is the isometric inclusion
\[
 i_e:(eAe)_*\longrightarrow A_*,\qquad
 (i_e\omega)(a)=\omega(eae).
\]
We shall use the following  result. Note that, the algebra $G$ below is not required to be acting on a separable Hilbert space.

\begin{theorem}\cite{HRS,PX,Sukochev,SX}\label{thm:hrs}
Let $0<  p\ne 2<\infty $. If $G$ is a finite von Neumann algebra with a
faithful normal finite trace $\sigma$, then
$\mathcal{C}_p\not\hookrightarrow L_p(G,\sigma)$, where $\mathcal{C}_p=\mathcal{C}_p(\ell_2)$.
\end{theorem}

We also use the standard fact that, for $1\le p<\infty$,
the $L_p$-space of a nonzero countably decomposable properly
infinite von Neumann algebra contains $\mathcal{C}_p$ isometrically;
see the introductory discussion in \cite{HRS}, and the
amplification theory in \cite{Takesaki}.
In particular, this applies to every nonzero {reduced algebra of a type III algebra}
used below. 

For a general von Neumann algebra $A$, $L_p(A)$ means the Haagerup space.
It is reflexive for $1<p<\infty$, with dual $L_q(A)$ when
$p^{-1}+q^{-1}=1$. Bounded left and right multiplication have
norm at most the operator norm of the multiplier. For a projection $e\in A$,
$eL_p(A)e$ identifies isometrically with $L_p(eAe)$.
We use these standard properties in the form given in
\cite[Sections 1 and 3]{PX}.

Let $(\mathcal A,\tau)$ be a finite von Neumann algebra with a faithful
normal tracial state. A family $(\mathcal A_i)_{i\in I}$ of unital
von Neumann subalgebras of $\mathcal A$ is \emph{free} with respect to
$\tau$ if
\[
 \tau(a_1\cdots a_n)=0
\]
whenever $n\ge1$, $a_j\in\mathcal A_{i_j}$, $\tau(a_j)=0$ for every
$j$, and $i_j\ne i_{j+1}$ for $1\le j<n$. A family of variables is
free if the unital von Neumann algebras generated by the individual
variables are free.

For finite von Neumann algebras $(A_1,\tau_1)$ and $(A_2,\tau_2)$
with faithful normal tracial states, write
\[
 (P,\tau)=(A_1,\tau_1)*(A_2,\tau_2),
\]
where $*$ stands for the \emph{tracial free product}. Thus, $P$ is a finite
von Neumann algebra with a faithful normal tracial state $\tau$,
containing unital normal copies of $A_1$ and $A_2$, such that
$P=W^*(A_1\cup A_2)$, $\tau|_{A_j}=\tau_j$ for $j=1,2$, and
$A_1,A_2$ are free with respect to $\tau$. These properties determine
$(P,\tau)$ up to a trace-preserving isomorphism fixing both copies.
The construction is obtained by taking the von Neumann algebra generated
by the reduced $C^*$-free product in the GNS representation of its
free product trace; see \cite[Lectures 6--7]{NS} for the $C^*$-construction.

More generally, let $(A_i,\tau_i)$ be finite von Neumann algebras
with faithful normal tracial states, containing a common unital
von Neumann subalgebra $F$, and suppose that $\tau_i|_F=\tau_F$.
Let $E_i$ be the trace-preserving conditional expectation from
$A_i$ onto $F$. We write
\[
 (P,E)=*_{F,i}(A_i,E_i),
\]
where $*_F$ stands for the \emph{reduced amalgamated free product}
over $F$; see {\cite[Section~1]{JPX2007}}. The algebra $P$ is generated as a von Neumann algebra
by the canonical copies of the $A_i$, and $E$ is a faithful normal conditional expectation
from $P$ onto $F$ such that $E|_{A_i}=E_i$ and
\[
 E(x_1\cdots x_n)=0
\]
whenever $x_j\in A_{i_j}$, $E_{i_j}(x_j)=0$ for every $j$, and
$i_j\ne i_{j+1}$ for $1\le j<n{<\infty}$. The state
$\tau_P=\tau_F\circ E$ is a faithful normal tracial state on $P$.

\subsection{Embeddings of row and column spaces}\label{sec:head}

For a finite tracial von Neumann algebra $(F,\tau_F)$
and $0\mathrel{<}p\le\infty$, let
$\mathcal{R}^{\row}_p(F)$ and $\mathcal{C}^{\col}_p(F)$ denote the completions of
finite sequences in $F$ with respect to the quasi-norms
\begin{equation}\label{eq:row-column-norms}
\begin{aligned}
 \left\|(a_i)\right\|_{\mathcal{R}^{\row}_p}
 =\left\|\left(\sum_i a_i a_i^*\right)^{1/2}\right\|_p,\qquad  \left\|(a_i)\right\|_{\mathcal{C}^{\col}_p}
 =\left\|\left(\sum_i a_i^*a_i\right)^{1/2}\right\|_p.
\end{aligned}
\end{equation}
{
Let $(e_{ij})_{i,j\ge1}$ be the standard matrix units of $B(\ell_2)$,
and equip $F\overline\otimes B(\ell_2)$ with the trace
$\tau_F\otimes\Tr$. For a finite sequence $a=(a_i)$ in $F$,
extended by zero, define
\begin{equation}\label{eq:row-column-identification}
 a\longmapsto T_r(a):=
 \begin{pmatrix}
  a_1&a_2&\cdots\\
  0&0&\cdots\\
  \vdots&\vdots&\ddots
 \end{pmatrix},\qquad
 a\longmapsto T_c(a):=
 \begin{pmatrix}
  a_1&0&\cdots\\
  a_2&0&\cdots\\
  \vdots&\vdots&\ddots
 \end{pmatrix}.
\end{equation}
Since the $(1,1)$-entries of $T_r(a)T_r(a)^*$ and $T_c(a)^*T_c(a)$ are
$\sum_i a_i a_i^*$ and $\sum_i a_i^*a_i$, respectively, we have
\begin{equation}\label{eq:row-column-isometries}
 \left\|T_r(a)\right\|_{L_p(F\overline\otimes B(\ell_2))}
 =\left\|a\right\|_{\mathcal{R}^{\row}_p},\qquad
 \left\|T_c(a)\right\|_{L_p(F\overline\otimes B(\ell_2))}
 =\left\|a\right\|_{\mathcal{C}^{\col}_p}.
\end{equation}
The maps $T_r$ and $T_c$ identify the completions of the space of
finitely supported sequences $a=(a_i)_{i\ge1}$ in $F$ under the row and
column quasi-norms in \eqref{eq:row-column-norms} with the respective
closures of the finite matrix rows
and columns in $L_p(F\overline\otimes B(\ell_2))$; see also
\cite[Section~1, p.~670]{PX1997} for $1\le p<\infty$.
For $p=\infty$, we take the closures in the operator norm.
}

\begin{lemma}\label{lem:row}
Let $(F,\tau_F)$ have a faithful normal tracial state, and let
$0\mathrel{<}p\le\infty$. Define $\gamma_p=3^{1-2/p}$ for $0<p<2$ and
$\gamma_p=1$ for $2\le p\le\infty$. There are a finite von Neumann
algebra $P$ with a faithful normal tracial state $\tau_P$ and a linear map
$J_p:\mathcal{R}^{\row}_p(F)\mathrel{\to} L_p(P,\tau_P)$ such that
\begin{equation}\label{eq:row-bounds}
 \gamma_p\left\|a\right\|_{\mathcal{R}^{\row}_p}
 \le\left\|J_pa\right\|_{L_p(P,\tau_P)}
 \le3\left\|a\right\|_{\mathcal{R}^{\row}_p}.
\end{equation}
The same assertion holds for the column map
$J_p^{\col}(a):=J_p((a_i^*)_i)^*$, where $a=(a_i)_i\in\mathcal C_p^{\col}(F)$.
For $p=\infty$, the row map $J_\infty$ and the column map
$J_\infty^{\col}$ extend with the bounds in \eqref{eq:row-bounds} to all bounded
first rows and first columns in $F\overline\otimes B(\ell_2)$, respectively.
\end{lemma}

\begin{proof}
Let $m$ be normalized Lebesgue measure on $\mathbb T$.
For each $i\ge1$, take
\[
 (A_i,\tau_i)=(F,\tau_F)*(L_\infty(\mathbb T),m).
\]
Let $\iota_i:L_\infty(\mathbb T,m)\to A_i$ be the canonical
unital $*$-embedding given by \cite[Theorem~7.9]{NS}. Define
\[
 w\in L_\infty(\mathbb T,m),\qquad
 w(z)=z\quad(z\in\mathbb T),\qquad u_i=\iota_i(w).
\]
Thus, $u_i^*u_i=u_i u_i^*=1_{A_i}$. Let $E_i$ be the
trace-preserving conditional expectation from $A_i$ onto $F$.
Since $E_i$ is $F$-bimodular and trace-preserving, it follows that 
\begin{equation}\label{eq:expectation-pairing}
 \tau_F(cE_i(x))=\tau_F(E_i(cx))=\tau_i(cx)
 \qquad(c\in F,\ x\in A_i).
\end{equation}
Let $P$ be the von Neumann algebra defined by
\[
 (P,E)=*_{F,i\ge1}(A_i,E_i),
\]
where $E$ is the canonical conditional expectation from $P$ onto $F$.
The state $\tau_P=\tau_F\circ E$ is faithful, normal and tracial.
For a finite sequence
$a=(a_i)$ in $F$, put $$J(a):=\sum_i a_i u_i.$$

The subalgebras $F$ and $\iota_i(L_\infty(\mathbb T,m))$ of $A_i$
are free with respect to $\tau_i$; see \cite[Theorem~7.9(ii)]{NS}.
{
Write $B_i:=\iota_i(L_\infty(\mathbb T,m))$. For $b,c\in F$, put
$b_0:=b-\tau_F(b)1$ and $c_0:=c-\tau_F(c)1$.
Since $\tau_i|_F=\tau_F$ and
$\tau_i(u_i)=\int_{\mathbb T}z\,dm(z)=0$, \textcolor{black}{we have}
\[
\begin{gathered}
 c_0,b_0\in F,\qquad u_i^*,u_i\in B_i,\\
 \tau_i(c_0)=\tau_i(b_0)=\tau_i(u_i^*)=\tau_i(u_i)=0.
\end{gathered}
\]
In particular, the freeness of $F$ and $B_i$ gives
$\tau_i(c_0u_i)=0$. Hence, \textcolor{black}{we have} 
\[
\begin{aligned}
 \tau_F(cE_i(u_i))
 \overset{\eqref{eq:expectation-pairing}}{=}\tau_i(cu_i)
 =\tau_i(c_0u_i)+\tau_F(c)\tau_i(u_i)=0\qquad(c\in F).
\end{aligned}
\]
Taking $c=E_i(u_i)^*$ and using the faithfulness of $\tau_F$,
we obtain $E_i(u_i)=0$.
}

By the freeness of $F$ and $B_i$ with respect to $\tau_i$, we have 
\[
\begin{aligned}
 \tau_i\!\left(
 \underbrace{c_0}_{\in F}\,
 \underbrace{u_i^*}_{\in B_i}\,
 \underbrace{b_0}_{\in F}\,
 \underbrace{u_i}_{\in B_i}\right)&=0\mbox{ and }
 \tau_i\!\left(
 \underbrace{u_i^*}_{\in B_i}\,
 \underbrace{b_0}_{\in F}\,
 \underbrace{u_i}_{\in B_i}\right)&=0.
\end{aligned}
\]
Substituting $b=b_0+\tau_F(b)1$ and then $c=c_0+\tau_F(c)1$,
and using $u_i^*u_i=1$, yields
\begin{equation}\label{eq:centered-product-expansion}
\begin{aligned}
 cu_i^*bu_i
 =cu_i^*b_0u_i+\tau_F(b)c=c_0u_i^*b_0u_i+\tau_F(c)u_i^*b_0u_i+\tau_F(b)c.
\end{aligned}
\end{equation}
Taking the trace {$\tau_i$} in \eqref{eq:centered-product-expansion}, using
$\tau_i(c)=\tau_F(c)$, and applying \eqref{eq:expectation-pairing}
with $x=u_i^*bu_i$, we obtain
\[
\begin{aligned}
 \tau_F(cE_i(u_i^*bu_i))
 =\tau_i(cu_i^*bu_i)=\tau_F(c)\tau_F(b)+\tau_i(c_0u_i^*b_0u_i)+\tau_F(c)\tau_i(u_i^*b_0u_i)
 =\tau_F(c)\tau_F(b).
\end{aligned}
\]
Thus,  $d:=E_i(u_i^*bu_i)-\tau_F(b)1\in F$ satisfies
$\tau_F(cd)=0$ for every $c\in F$. Taking $c=d^*$ and using
faithfulness gives $d=0$, that is,
\[
 E_i(u_i^*bu_i)=\tau_F(b)1\qquad(b\in F).
\]
Using $E|_{A_i}=E_i$ and the $F$-bimodule property of $E_i$,
for $x_i=a_i u_i$ with $a_i\in F$, we obtain
$E(x_i){=E_i(a_i u_i)=a_iE_i(u_i)}=0$ and
\begin{align}\label{Exixi}
 E(x_ix_i^*)=a_i a_i^*,\qquad
 E(x_i^*x_i)=\tau_F(a_i^*a_i)1.
\end{align}
Since the subalgebras $(A_i)$ are free with amalgamation over $F$
with respect to $E$, 
\cite[Proposition~7.1]{Junge} {together with \eqref{Exixi} above } gives,
for any $n\ge1$ and $a=(a_i)_{i=1}^n\in F^n$,
\begin{equation}\label{eq:row-infty}
\begin{aligned}
 \left\|J(a)\right\|_\infty
 &\le\sup_i\left\|a_i\right\|_\infty
   +\left(\sum_i\tau_F(a_i^*a_i)\right)^{1/2}
   +\left\|\left(\sum_i a_i a_i^*\right)^{1/2}\right\|_\infty\\
 &{\le2\left\|a\right\|_{\mathcal{R}^{\row}_\infty}
   +\tau_F\left(\sum_i a_i^*a_i \right)^{1/2}}{\overset{\tau_F(1)=1}{\le}
   2\left\|a\right\|_{\mathcal{R}^{\row}_\infty}
   +\left\|\sum_i a_i a_i^*\right\|_\infty^{1/2}}
   =3\left\|a\right\|_{\mathcal{R}^{\row}_\infty}.
\end{aligned}
\end{equation}
Let $(e_{ij})_{i,j=1}^n$ be the standard matrix units and set
$R=\sum_{i=1}^n e_{1i}\otimes a_i\in \mathbb{M}_n(F)$.
By the polar decomposition \cite[Section~I.5.2]{Blackadar},
\[
 R=(e_{11}\otimes h)V,\qquad
 VV^*=e_{11}\otimes s(h),\qquad
 h=\left(\sum_i a_i a_i^*\right)^{1/2},
\]
where $V\in \mathbb{M}_n(F)$ is a partial isometry and $s(h)$ is the support
projection of $h$. Thus, $V=\sum_i e_{1i}\otimes v_i$ with $v_i\in F$, and
\[
 a_i=hv_i,\qquad \sum_i v_i v_i^*=s(h)\le1.
\]
Set $v:=(v_i)_{i=1}^n\in F^n$. Then
$J(v)=\sum_{i=1}^n v_i u_i$ and $\left\|v\right\|_{\mathcal{R}^{\row}_\infty}\le1$.
Hence, we have 
\begin{equation}\label{eq:row-upper}
 \left\|J(a)\right\|_p=\left\|hJ(v)\right\|_p
 \stackrel{\eqref{eq:row-infty}}{\le}3\left\|h\right\|_p=3\left\|a\right\|_{\mathcal{R}^{\row}_p}
 \qquad(0\mathrel{<}p\le\infty).
\end{equation}

For finite sequences $a,b$ in $F$, amalgamated freeness gives
$E(u_i u_j^*)=\delta_{ij}1$. By $F$-bimodularity, we have
\begin{equation}\label{eq:row-expectation}
 E(J(a)J(b)^*)=
 \sum_{i,j}a_i E(u_i u_j^*)b_j^*=\sum_i a_i b_i^* .
\end{equation}
In particular, $E(J(a)J(a)^*)=h^2$ and $E(J(v)J(v)^*)=s(h)$.
For $0<p<2$, set $U=J(v)$, so that $J(a)=hU$ and
$\|U\|_\infty\le3$ by \eqref{eq:row-infty}.
Applying Hansen's inequality \cite{Hansen1980} to the contraction $U/3$,
the positive operator $h^2$ and the function $t\mapsto t^{p/2}$, we obtain
\[
 (U^*h^2U)^{p/2}\ge 3^{p-2}U^*h^pU.
\]
Since $\tau_P=\tau_F\circ E$, traciality and \eqref{eq:row-expectation} imply
\[
 \|J(a)\|_{L_p(P,\tau_P)}^p
 \ge 3^{p-2}\tau_P(U^*h^pU)
 =3^{p-2}\tau_F\bigl(h^pE(UU^*)\bigr)
 =3^{p-2}\|h\|_{L_p(F,\tau_F)}^p .
\]
For $2\le p\le\infty$, the trace-preserving conditional expectation $E$
is contractive from $L_{p/2}(P,\tau_P)$ to $L_{p/2}(F,\tau_F)$
\cite[Example~5.8]{HJX}, with $p/2=\infty$ for $p=\infty$.
Thus, by \eqref{eq:row-expectation},
\[
 \|h\|_{L_p(F,\tau_F)}^2
 =\|E(J(a)J(a)^*)\|_{L_{p/2}(F,\tau_F)}
 \le\|J(a)J(a)^*\|_{L_{p/2}(P,\tau_P)}
 =\|J(a)\|_{L_p(P,\tau_P)}^2 .
\]
Together with \eqref{eq:row-upper}, these inequalities prove
\eqref{eq:row-bounds} on finite rows. By completeness, $J$ extends to
$\mathcal R_p^{\row}(F)$ with the same bounds.

For every bounded first row $R$, the bounds \eqref{eq:row-bounds} for
$J_\infty(R):=J_2(R)$ follow by a standard approximation argument
using \eqref{eq:row-infty} and \eqref{eq:row-expectation} on the finite
truncations $R_k\to R$ in $L_2(F\overline\otimes B(\ell_2))$.
The bounds for $J_p^{\col}$ follow by taking adjoints.
\end{proof}

{The next proposition is a consequence of  the row and column
embeddings of Lemma~\ref{lem:row}.}

\begin{proposition}\label{lem:head}
Let $(F,\tau_F)$ be a finite von Neumann algebra with a faithful
normal tracial state. Put $N=F\overline\otimes B(\ell_2)$, with trace
$\tau_F\otimes\Tr$.
Let $f_n=1\otimes P_n$, with $P_n$ the first $n$-coordinate
projection, and put $q_n=1-f_n$. For $0\mathrel{<}p\mathrel{\le}\infty$, set
\[
 H_{n,p}=\{y\in L_p(N):q_nyq_n=0\},\qquad
 r=\min\{1,p\},\qquad \kappa_p=2^{1/p-1/r}\gamma_p ,
\]
where $1/\infty=0$ and $\gamma_p$ is defined in Lemma~\ref{lem:row}.
There are a finite von Neumann algebra $G_n$ with a faithful normal
finite trace $\sigma_n$ and a linear map
$\Phi_n:H_{n,p}\mathrel{\to} L_p(G_n,\sigma_n)$ such that
\begin{equation}\label{eq:head-bounds}
 \kappa_p\|y\|_{L_p(N)}
 \le\|\Phi_n(y)\|_{L_p(G_n,\sigma_n)}
 \le3\cdot 2^{1/p}\|y\|_{L_p(N)}
 \qquad(y\in H_{n,p}).
\end{equation}
\end{proposition}

\begin{proof}Define $F_n:=F\overline\otimes\mathbb{M}_n$. Observe that
{(where $\cong$ denotes a unitary isomorphism of Hilbert
spaces or a normal $*$-isomorphism of von Neumann algebras, as appropriate)}
\[
\begin{gathered}
 \ell_2\cong\mathbb C^n\otimes\ell_2 
 \\
 N\cong F_n\overline\otimes B(\ell_2),\qquad
 f_n\longleftrightarrow1_{F_n}\otimes e_{11},\\
 \tau_{F_n}:=\tau_F\otimes\Tr_n,\qquad
 \tau_{F_n}(1_{F_n})=n.
\end{gathered}
\]
For $y\in H_{n,p}$, $q_nyq_n=0$ yields
\begin{equation}\label{eq:head-decomposition}
 y=f_ny+q_nyf_n.
\end{equation}
In the matrix representation \eqref{eq:row-column-identification} over
$F_n$, $f_ny$ is a first row and $q_nyf_n$ is a first column.
For $p<\infty$, let $s_k=1_{F_n}\otimes\sum_{i=1}^k e_{ii}$.
Then $f_nys_k\to f_ny$ and $s_kq_nyf_n\to q_nyf_n$ in $L_p(N)$ as $k\to\infty$
\cite[Remark~4.2, p.~171]{RX}. Approximating the finitely many entries
of $f_nys_k$ and $s_kq_nyf_n$ in $L_p(F_n,\tau_{F_n})$ by elements of $F_n$,
we obtain $f_ny\in\mathcal R_p^{\row}(F_n)$ and $q_nyf_n\in\mathcal C_p^{\col}(F_n)$.
For the case when $p=\infty$, we use the bounded row and column maps of Lemma~\ref{lem:row}.

Let $J_r,J_c$ be the row and column embeddings of Lemma~\ref{lem:row}
for $(F_n,\tau_{F_n}/n)$ and $(P,\tau_P)$. Define
\[
 G_n=P\oplus P,\qquad \sigma_n=n(\tau_P\oplus\tau_P),\qquad
 \Phi_n(y)=\bigl(J_r(f_ny),J_c(q_nyf_n)\bigr).
\]
Rescaling both traces by $n$ in \eqref{eq:row-bounds} and using the
isometries \eqref{eq:row-column-isometries}, we obtain
\begin{equation}\label{eq:head-row-column-bounds}
\begin{aligned}
 \gamma_p\|f_ny\|_{L_p(N)}
 &\le\|J_r(f_ny)\|_{L_p(P,n\tau_P)}\le3\|f_ny\|_{L_p(N)},\\
 \gamma_p\|q_nyf_n\|_{L_p(N)}
 &\le\|J_c(q_nyf_n)\|_{L_p(P,n\tau_P)}\le3\|q_nyf_n\|_{L_p(N)} .
\end{aligned}
\end{equation}
For the case when $p<\infty$, the decomposition \eqref{eq:head-decomposition}
and the $r$-triangle inequality imply
$\|y\|_{L_p(N)}^r\le\|f_ny\|_{L_p(N)}^r+\|q_nyf_n\|_{L_p(N)}^r$.
Applying \cite[(2.3)]{KPR} to the norms of $f_ny$ and $q_nyf_n$ yields the lower bound in
\begin{equation}\label{eq:head-sum-bounds}
 2^{1-p/r}\|y\|_{L_p(N)}^p
 \le \|f_ny\|_{L_p(N)}^p+\|q_nyf_n\|_{L_p(N)}^p
 \le 2\|y\|_{L_p(N)}^p .
\end{equation}

By the definition of $\sigma_n$, we have
\[
 \|\Phi_n(y)\|_{L_p(G_n,\sigma_n)}^{p}
 =\|J_r(f_ny)\|_{L_p(P,n\tau_P)}^p+\|J_c(q_nyf_n)\|_{L_p(P,n\tau_P)}^p .
\]
Applying \eqref{eq:head-row-column-bounds} to the two summands in the
above formula and then using \eqref{eq:head-sum-bounds}, we obtain
\[
 2^{1-p/r}\gamma_p^p\|y\|_{L_p(N)}^p
 \le \|\Phi_n(y)\|_{L_p(G_n,\sigma_n)}^p
 \le 2\cdot3^p\|y\|_{L_p(N)}^p .
\]
Since $\kappa_p=2^{1/p-1/r}\gamma_p$, this proves \eqref{eq:head-bounds}
for the case when $p<\infty$.

For the case when $p=\infty$, we have $\gamma_\infty=1$ and $G_n=P\oplus P$
has the maximum norm. By \eqref{eq:head-decomposition} and
\eqref{eq:head-row-column-bounds}, we obtain
\[
\begin{aligned}
 \tfrac12\|y\|_\infty
 &\le \max\{\|f_ny\|_\infty,\|q_nyf_n\|_\infty\}\\
 &\le \max\{\|J_r(f_ny)\|_\infty,\|J_c(q_nyf_n)\|_\infty\}
     =\|\Phi_n(y)\|_\infty\\
 &\le 3\max\{\|f_ny\|_\infty,\|q_nyf_n\|_\infty\}\le3\|y\|_\infty .
\end{aligned}
\]
This proves \eqref{eq:head-bounds} with $\kappa_\infty=1/2$.
\end{proof}

\section{\newtext{The key} result for the proof of Theorem~\ref{thm:unified}: the predual case}\label{sec:endpoint}

We begin with a local compression estimate for normal bounded maps
{$u:A\to B$ between von Neumann algebras $A$ and $B$}.
Together with Proposition~\ref{lem:head}, the estimate
{in Lemma~\ref{lem:tail} below} reduces
{the proof of Theorem~\ref{thm:main}}
to the non-embedding theorem of Haagerup, Rosenthal and
the second author~\cite{HRS}, see Theorem~\ref{thm:hrs} above.

\begin{lemma}\label{lem:tail}
Let $A,B$ be von Neumann algebras, let $u:A\to B$ be an ultraweakly continuous bounded complex-linear map, and let $q_n\in A$ be projections decreasing strongly to zero. For every nonzero projection $e\in B$ and every $\varepsilon>0$, there are a nonzero projection $f\le e$ and an integer $n$ such that
\begin{equation}\label{eq:tail}
 \sup\{\left\|f u(x)f\right\|_\infty:x\in q_nAq_n,\ \left\|x\right\|_\infty\le1\}<\varepsilon.
\end{equation}
\end{lemma}

\begin{proof}
Assume by contradiction that there exist a nonzero projection $e\in B$
and $\eta>0$ such that, for every nonzero projection $f\le e$ and
every $n\ge1$,
\[
 \sup_{\substack{x\in q_nAq_n\\ \left\|x\right\|_\infty\le1}}
 \left\|f u(x)f\right\|_\infty\ge\eta.
\]
Fix $0<\delta<\eta/8$.

We construct by induction integers $n_1<n_2<\cdots$, nonzero projections
\[
 e=f_0\ge f_1\ge f_2\ge\cdots,
\]
and contractions $x_k\in(q_{n_k}-q_{n_{k+1}})A(q_{n_k}-q_{n_{k+1}})$ such that
\begin{equation}\label{eq:preserved}
 f_k\Rea u(x_i)f_k\ge\delta f_k\qquad(1\le i\le k){,}
\end{equation}
{where $\Rea x:=(x+x^*)/2$ and
$\Ima x:=(x-x^*)/(2i)$ denote the real and imaginary parts of
a bounded operator $x$, respectively.}
Set $n_1=1$. Given $f_{k-1}$ and $n_k$, choose a contraction $x\in q_{n_k}Aq_{n_k}$ with
\[
 \left\|f_{k-1}u(x)f_{k-1}\right\|_\infty>4\delta.
\]
For every operator $a$, we have $\left\|a\right\|_\infty\le\left\|\Rea a\right\|_\infty+\left\|\Ima a\right\|_\infty$.
Hence, at least one of the projections
\[
\begin{gathered}
 1_{(2\delta,\infty)}\bigl(f_{k-1}\Rea u(x)f_{k-1}\bigr),\qquad
 1_{(2\delta,\infty)}\bigl(f_{k-1}\Rea u(-x)f_{k-1}\bigr),\\
 1_{(2\delta,\infty)}\bigl(f_{k-1}\Rea u(ix)f_{k-1}\bigr),\qquad
 1_{(2\delta,\infty)}\bigl(f_{k-1}\Rea u(-ix)f_{k-1}\bigr)
\end{gathered}
\]
is nonzero. Replacing $x$ by the corresponding element of
$\{x,-x,ix,-ix\}$   (we write $x$ for simplicity below), set
\[
 0\ne r:=1_{(2\delta,\infty)}
 \bigl(f_{k-1}\Rea u(x)f_{k-1}\bigr)\le f_{k-1}.
\]
Then
\begin{equation}\label{eq:signal}
 r\Rea u(x)r\ge2\delta r.
\end{equation}

For $m>n_k$, put
\begin{align}\label{def xm}
 x^{(m)}=(q_{n_k}-q_m)x(q_{n_k}-q_m).
\end{align}
The contractions $x^{(m)}$ converge   ultraweakly (indeed, in the strong-$*$ topology
{\cite[Chapter~II, Section~2]{Takesaki}})
to $x$ as $m\to \infty $. Choose a normal state $\psi$ on $B$ supported under $r$ such  that $\psi(r)=1$. Normality of $u$ gives
\[
 \psi\bigl(r\Rea u(x-x^{(m)})r\bigr)\longrightarrow0.
\]
Choose $m>n_k$ such that    $ \psi\bigl(r\Rea u(x-x^{(m)})r\bigr) < \delta$, and set
\[
 b=r\Rea u(x-x^{(m)})r\in(rBr)_{\mathrm{sa}}.
\]
The spectral projection
\[
 f_k=1_{(-\infty,\delta]}^{\,rBr}(b)
\]
is nonzero. Indeed,  otherwise $b\ge\delta r$, contradicting $\psi(b)<\delta$. 

Set $n_{k+1}=m$ and $x_k=x^{(m)}$. Inequality \eqref{eq:signal} further gives
\[
 f_k\Rea u(x_k)f_k
 =f_k\Rea u(x)f_k-f_kbf_k
 \ge\delta f_k.
\]
Since $f_k\le f_{k-1}$, all previous inequalities in \eqref{eq:preserved} remain valid, i.e.,
\(
 f_k\Rea u(x_i)f_k\ge\delta f_k, ~(1\le i<k).
\)

Observe that the  projections $q_{n_k}-q_{n_{k+1}}$, $k\ge 1$, are pairwise orthogonal. Hence, by \eqref{def xm}, for every finite integer $K$, we have 
\[
 \left\|\sum_{k=1}^Kx_k\right\|_\infty\le1.
\]
On the other hand, \eqref{eq:preserved} implies
\[
 f_K\Rea u\left(\sum_{k=1}^Kx_k\right)f_K\ge K\delta f_K.
\]
Since $f_K$ is a nonzero projection, it follows that  $\left\|u\right\|\ge K\delta$ for all $K\ge 1$, which is impossible.  
\end{proof}

The following Theorem is the key to prove  Theorem \ref{thm:unified} for the case of preduals of von Neumann algebras. 
\begin{theorem}\label{thm:main}
Let $M$ be a nonzero type III von Neumann algebra with separable predual,
let $(F,\tau_F)$ be a finite von Neumann algebra with faithful normal
tracial state, and put
\[
                     N=F\overline\otimes B(\ell_2).
\]
Then $M_*\not\hookrightarrow N_*$.
\end{theorem}

\begin{proof}
Suppose that $T:M_*\hookrightarrow N_*=L_1(N,\tau_N)$ and choose $c>0$ such that
\[
 \left\|T\omega\right\|_1\ge c\left\|\omega\right\|\qquad(\omega\in M_*).
\]
The adjoint $u=T^*:N\to M$ is bounded and ultraweakly continuous.

Use $f_n=1\otimes P_n$ as in Proposition~\ref{lem:head} and set $q_n=1-f_n$. By Lemma~\ref{lem:tail}, there are $n$ and a nonzero projection $f\in M$ such that
\begin{equation}\label{eq:dual-small}
 \left\|x\mapsto fT^*(q_nxq_n)f\right\|_{N\to fMf}<c/2.
\end{equation}
The canonical inclusion $i_f:(fMf)_*\hookrightarrow M_*$ is
isometric and is the preadjoint of $a\mapsto faf$; see
\cite[Sections~III.2.1--III.2.2]{Blackadar}.
Define
\[
 Q_n:L_1(N)\longrightarrow L_1(N),\qquad
 Q_n(y):=q_nyq_n.
\]
Then the map in \eqref{eq:dual-small} is exactly $(Q_nTi_f)^*$
{(indeed, for $x\in N$ and $\omega\in(fMf)_*$, \textcolor{black}{we have}
\[
\begin{aligned}
 \langle\omega,(Q_nTi_f)^*(x)\rangle
 &=\langle Q_nTi_f\omega,x\rangle\\
 &=\tau_N\bigl(q_n(Ti_f\omega)q_nx\bigr)
 =\tau_N\bigl((Ti_f\omega)q_nxq_n\bigr)\\
 &=\langle Ti_f\omega,q_nxq_n\rangle
 =\langle i_f\omega,T^*(q_nxq_n)\rangle\\
 &=\langle\omega,fT^*(q_nxq_n)f\rangle.
\end{aligned}, 
\]
i.e.,  $(Q_nTi_f)^*(x)=fT^*(q_nxq_n)f$).}
Thus, {
by \cite[Theorem~4.10, p.~98]{Rudin1991} and \eqref{eq:dual-small},}
\[
 \left\|Q_nTi_f\right\|
 {=\left\|(Q_nTi_f)^*\right\|}<c/2.
\]
For every $\omega\in(fMf)_*$, we obtain
\begin{equation}\label{eq:predual-head-lower}
 \left\|(I-Q_n)Ti_f\omega\right\|_1
 \ge\left\|Ti_f\omega\right\|_1-\left\|Q_nTi_f\omega\right\|_1
 \ge(c/2)\left\|\omega\right\|.
\end{equation}
{Since
$q_n(I-Q_n)(y)q_n=0$ for every $y\in L_1(N)$, it follows that the image of $(I-Q_n)Ti_f$ lies in $H_{n,1}=\{y\in L_1(N):q_nyq_n=0\}$. }
Apply Proposition~\ref{lem:head} with $p=1$ to obtain a mapping
\[
 \Phi_n:H_{n,1}\hookrightarrow L_1(G_n,\sigma_n),
\]
where $G_n$ is a finite von Neumann algebra and $\sigma_n$ is a faithful
normal finite trace. Define the bounded linear map
\[
 S:=\Phi_n\circ(I-Q_n)\circ T\circ i_f:
 (fMf)_*\longrightarrow L_1(G_n,\sigma_n).
\]
For every $\omega\in(fMf)_*$, we have 
\[
\begin{aligned}
 \left\|S\omega\right\|_1
 \overset{\eqref{eq:head-bounds}}{\ge}
 \frac16\left\|(I-Q_n)Ti_f\omega\right\|_1\overset{\eqref{eq:predual-head-lower}}{\ge}
 \frac{c}{12}\left\|\omega\right\|.
\end{aligned}
\]
Thus, $S:(fMf)_*\hookrightarrow L_1(G_n,\sigma_n)$.

The corner $fMf$ is nonzero and countably decomposable.
Since $M$ is of type III, $fMf$ is also of type III; see
\cite[Definitions~III.1.3.1 and~III.1.4.2]{Blackadar}.
The predual $(fMf)_*$ therefore contains $\mathcal{C}_1$ isometrically,  see Section~\ref{sec:prelim} above. Restricting $S$ to this copy of
$\mathcal{C}_1$ gives $\mathcal{C}_1\hookrightarrow L_1(G_n,\sigma_n)$,
contrary to Theorem~\ref{thm:hrs}.
\end{proof}

\section{The key result for the proof of Theorem~\ref{thm:unified}: the case \texorpdfstring{$0<p<2$}{0<p<2}}\label{sec:array-proof}

Throughout this section, we always assume that $0<p<2$.

\begin{theorem}\label{thm:interior}
Let $M$ be a nonzero hyperfinite type III von Neumann algebra with
separable predual, and let $(F,\tau_F)$ be a finite von Neumann algebra
with a faithful normal tracial state. Then we have
\[
 L_p(M)\not\hookrightarrow L_p(F\overline\otimes B(\ell_2)).
\]
\end{theorem}

\subsection{Equiintegrability and reduction to a finite trace}

Let \(S\) be a hyperfinite type III von Neumann algebra with separable
predual and let \(\varphi\) be a faithful normal state on \(S\). Define
\(h=D_\varphi^{1/p}\). In particular, \(\|h\|_{L_p(S)}=1\). Let
\(N=F\overline\otimes B(\ell_2)\) with trace
\(\tau_N=\tau_F\otimes\Tr\), where \(\tau_F\) is a faithful normal
tracial state on \(F\). Suppose that a linear map
\(T:L_p(S)\to L_p(N,\tau_N)\) satisfies
\begin{equation}\label{quasi:embedding}
 c\|x\|_{L_p(S)}\le\|Tx\|_{L_p(N,\tau_N)}
 \le C\|x\|_{L_p(S)}\qquad(x\in L_p(S))
\end{equation}
for some \(0<c\le C<\infty\). Set \(r=\min\{1,p\}\). For a von Neumann
algebra \(D\), write \(D_1=\{a\in D:\|a\|_\infty\le1\}\) and
\(\mathcal U(D)=\{u\in D:u^*u=uu^*=1_D\}\).

Let \(V\) be a von Neumann algebra. A bounded set \(K\subset L_p(V)\)
is \(p\)-equiintegrable if every net of projections \(e_\alpha\in V\)
converging ultraweakly to zero satisfies
\begin{equation}\label{quasi:equi-definition}
 \lim_\alpha\sup_{y\in K}\|e_\alpha y e_\alpha\|_{L_p(V)}=0 .
\end{equation}

\begin{lemma}\label{lem:quasi-equi-maps}
Let \(V,W\) be von Neumann algebras and let \(L:L_p(V)\to L_p(W)\)
be bounded and linear. If \(K\subset L_p(V)\) is bounded and
\(p\)-equiintegrable, then \(L(K)\) is \(p\)-equiintegrable in \(L_p(W)\).
In particular, for each \(x\in L_p(V)\), the sets
\begin{equation}\label{quasi:one-sided}
 \{L(ax):a\in V_1\},\qquad \{L(xa):a\in V_1\}
 \quad\text{are \(p\)-equiintegrable in \(L_p(W)\).}
\end{equation}
\end{lemma}

\begin{proof}
Suppose that \(K\) is \(p\)-equiintegrable but \(L(K)\) fails
\eqref{quasi:equi-definition} in \(L_p(W)\).
By \cite[Corollary~4.9, p.~176]{RX}, there exist \(x_j\in K\), mutually
orthogonal projections \(e_j\in W\), and \(\delta_0>0\) such that
\(\|e_jL(x_j)e_j\|_p\ge\delta_0\) for every \(j\).
After taking a subsequence, the splitting lemma
\cite[Corollary~4.8, p.~176]{RX} yields the decomposition
\[
 L(x_j)=z_j+d_j\in L_p(W),
\]
where
\((z_j)\) is \(p\)-equiintegrable and \(d_j=g_jd_jg_j\) for pairwise orthogonal projections \(g_j\in W\). Since \(e_j\to0\) ultraweakly,
we obtain \(\|e_jz_je_j\|_p\to0\). By the \(r\)-triangle inequality, we have
\[
 \|d_j\|_p^r
 \ge \|e_jd_je_j\|_p^r
 \ge \|e_jL(x_j)e_j\|_p^r-\|e_jz_je_j\|_p^r
 \ge \delta_0^r-\|e_jz_je_j\|_p^r.
\]
Passing to a subsequence if necessary, we have \(\|d_j\|_p\ge\delta_0/2\) for all \(j\).

Let $(\varepsilon_j)_{j\ge1}$ be independent symmetric signs. For a
$p$-equiintegrable sequence $(w_j)_{j\ge1}\subset L_p(V)$,
\cite[Lemma~5.2, p.~182]{RX} gives
\begin{equation}\label{quasi:sign-rate}
 k^{-1/p}\mathbb E\left\|\sum_{j=1}^k\varepsilon_j w_j\right\|_{L_p(V)}
 \longrightarrow0\qquad(k\to\infty).
\end{equation}
Apply \eqref{quasi:sign-rate} to $(x_j)_{j\ge1}\subset L_p(V)$
and $(z_j)_{j\ge1}\subset L_p(W)$. By disjointness of
\((d_j)\), the \(r\)-triangle inequality, and concavity of \(s^r\), we obtain
\[
 \begin{aligned}
 (\delta_0/2)^r k^{r/p}
 &\le \left(\sum_{j=1}^k\|d_j\|_p^p\right)^{r/p}
 =\mathbb E\left\|\sum_{j=1}^k\varepsilon_j d_j\right\|_p^r\\
 &\le \|L\|^r\left(\mathbb E\left\|\sum_{j=1}^k\varepsilon_jx_j\right\|_p\right)^r
   +\left(\mathbb E\left\|\sum_{j=1}^k\varepsilon_jz_j\right\|_p\right)^r
 =o(k^{r/p}).
 \end{aligned}
\]
This contradiction proves that \(L(K)\) is \(p\)-equiintegrable.
Assertion~\eqref{quasi:one-sided} now follows because both \(V_1x\) and
\(xV_1\) are \(p\)-equiintegrable (see \cite[Remark~4.2, p.~171]{RX}).
\end{proof}

\begin{proposition}\label{prop:quasi-finite-target}
Suppose that \(T:L_p(S)\to L_p(N,\tau_N)\) satisfies
\eqref{quasi:embedding}, and let \(h=D_\varphi^{1/p}\).
There exist a finite von Neumann algebra \(G\) with a faithful normal
tracial state \(\tau_G\), a bounded linear map \(B:L_p(S)\to L_p(G,\tau_G)\)
and a constant \(\delta>0\) such that
\[
 \|B(ah)\|_{L_p(G,\tau_G)}\ge\delta
 \qquad(a\in S_1,\ \|ah\|_{L_p(S)}^p\ge1/4).
\]
\end{proposition}

\begin{proof}
For $m\ge1$, let $q_m=1-1\otimes P_m$ as in Proposition~\ref{lem:head}.

By Lemma~\ref{lem:quasi-equi-maps}, applied with \(L=T\) and \(x=h\),
we obtain that $T(S_1h)$ is $p$-equiintegrable. Since $q_m\downarrow0$ strongly, choose
$m$ such that
\[
 \sup_{a\in S_1}\|q_m T(ah)q_m\|_p^r<\tfrac12 c^r 4^{-r/p}.
\]
For $a\in S_1$ with $\|ah\|_p^p\ge1/4$, we have $\|T(ah)\|_p^r\ge c^r4^{-r/p}$ by \eqref{quasi:embedding}.
By the $r$-triangle inequality, we obtain
\begin{equation}\label{quasi:compressed-orbit-lower}
 \|T(ah)-q_mT(ah)q_m\|_{L_p(N,\tau_N)}^r
 \ge \|T(ah)\|_p^r-\|q_mT(ah)q_m\|_p^r
 \ge \frac{c^r4^{-r/p}}{2}.
\end{equation}
For each $x\in L_p(S)$, we have $Tx-q_mT(x)q_m\in H_{m,p}$ since
$q_m(Tx-q_mT(x)q_m)q_m=0$. Proposition~\ref{lem:head} with $n=m$ supplies the
embedding $\Phi_m:H_{m,p}\hookrightarrow L_p(G_m,\sigma_m)$ satisfying
\eqref{eq:head-bounds}. Equip $G=G_m$ with the tracial state
$\tau_G=\sigma_m/\sigma_m(1)$ and define
\[
 B:L_p(S)\longrightarrow L_p(G,\tau_G),\qquad
 Bx=\sigma_m(1)^{1/p}\Phi_m\bigl(Tx-q_mT(x)q_m\bigr).
\]
For $z\in L_p(G_m,\sigma_m)$, we have
$\|\sigma_m(1)^{1/p}z\|_{L_p(G,\tau_G)}^p
=\sigma_m(1)\tau_G(|z|^p)=\sigma_m(|z|^p)$.
For each $x\in L_p(S)$, the $r$-triangle inequality and contractivity
of compression by $q_m$ yield
\[
 \|Tx-q_mT(x)q_m\|_p^r
 \le \|Tx\|_p^r+\|q_mT(x)q_m\|_p^r
 \le 2\|Tx\|_p^r .
\]
Using this estimate and the definition of $B$, we obtain
\[
\begin{aligned}
 \|Bx\|_{L_p(G,\tau_G)}
 &=\|\Phi_m\bigl(Tx-q_mT(x)q_m\bigr)\|_{L_p(G_m,\sigma_m)}\\
 &\overset{\eqref{eq:head-bounds}}{\le}
 3\cdot2^{1/p}\|Tx-q_mT(x)q_m\|_{L_p(N,\tau_N)}\\
 &\le 3\cdot2^{1/p+1/r}\|Tx\|_{L_p(N,\tau_N)}
 \overset{\eqref{quasi:embedding}}{\le}3\cdot2^{1/p+1/r}C\|x\|_{L_p(S)} .
\end{aligned}
\]
Thus, $\|B\|\le 3\cdot2^{1/p+1/r}C$.
For $a\in S_1$ with $\|ah\|_{L_p(S)}^p\ge1/4$, we obtain
\begin{equation}\label{quasi:finite-orbit}
 \begin{aligned}
 \|B(ah)\|_{L_p(G,\tau_G)}
 &=\|\Phi_m\bigl(T(ah)-q_mT(ah)q_m\bigr)\|_{L_p(G_m,\sigma_m)}\\
 &\overset{\eqref{eq:head-bounds}}{\ge}
   \kappa_p\|T(ah)-q_mT(ah)q_m\|_{L_p(N,\tau_N)}\\
 &\overset{\eqref{quasi:compressed-orbit-lower}}{\ge}
   \kappa_p\bigl(c^r4^{-r/p}/2\bigr)^{1/r}
   =2^{-1/r}4^{-1/p}\kappa_p c=:\delta>0 .
 \end{aligned}
\end{equation}
\end{proof}

\subsection{A unitary average in a finite tracial algebra}\label{sec:truncated-averages}

\begin{lemma}\label{lem:quasi-spectral-tails}
If $P$ is a von Neumann algebra with a faithful normal tracial state $\tau$ and
$K\subset L_p(P,\tau)$ is bounded and $p$-equiintegrable, then we have
\begin{equation}\label{quasi:tail-criterion}
 \lim_{R\to\infty}\sup_{Y\in K}\|Y1_{(R,\infty)}(|Y|)\|_p=0 .
\end{equation}
\end{lemma}

\begin{proof}
For \(Y\in K\) and \(R>0\), set $e=1_{(R,\infty)}(|Y|)$ and
$g=e\vee s\bigl((Ye)(Ye)^*\bigr)$, where $s\bigl((Ye)(Ye)^*\bigr)$ is the
left support of $Ye$. Then we have $Ye=gYge$ and
\[
 \tau(g)\le 2\tau(e)
 \overset{R^p e\le |Y|^p e}{\le}
 2R^{-p}\sup_{Z\in K}\|Z\|_p^p \longrightarrow 0
 \quad\text{as } R\to\infty.
\]
By the $p$-equiintegrability of $K$ in \eqref{quasi:equi-definition},
we have $\sup_{Z\in K}\|gZg\|_p\to0$
as $\tau(g)\to0$ (see
\cite[proof of Lemma~2.12, p.~261]{Randrianantoanina2002}
and, for $p\ge1$, \cite[Corollary~5.7.5, p.~358]{DPS2023}).
By $\|Ye\|_p=\|gYge\|_p\le\|gYg\|_p$ and the above bound on $\tau(g)$,
we prove \eqref{quasi:tail-criterion}.
\end{proof}

\begin{proposition}\label{prop:quasi-average}
Let \(B:L_p(S)\to L_p(G,\tau_G)\) and \(\delta>0\) be as in
Proposition~\ref{prop:quasi-finite-target}. Let \(A\) be a unital hyperfinite
von Neumann subalgebra of the centralizer
\[
 S_\varphi=\{a\in S:\varphi(ax)=\varphi(xa)\ \text{for all }x\in S\}.
\]
There exist a finite von Neumann algebra \(Q\) with a faithful normal
tracial state \(\tau_Q\) and a bounded linear map \(J:L_p(S)\to L_p(Q,\tau_Q)\)
with the following properties.
\begin{enumerate}
\item For each \(u\in\mathcal U(S)\) and \(v\in\mathcal U(A)\), there is a
\(\tau_Q\)-preserving \(*\)-automorphism \(\alpha_{u,v}\) of \(Q\) such that
\(J(uxv)=(\alpha_{u,v})_p(Jx)\) for every \(x\in L_p(S)\), where
\((\alpha_{u,v})_p\) is the induced isometry of \(L_p(Q,\tau_Q)\).
\item For every \(b\in S_1\) with \(\|bh\|_{L_p(S)}^p\ge1/4\),
we have \(\|J(bh)\|_{L_p(Q,\tau_Q)}\ge\delta\).
\item The set \(J(S_1hS_1)=\{J(ahb):a,b\in S_1\}\) is
\(p\)-equiintegrable in \(L_p(Q,\tau_Q)\).
\end{enumerate}
\end{proposition}

\begin{proof}
Let $(\sigma_t^\varphi)_{t\in\mathbb R}$ denote the modular automorphism
group of the faithful normal state $\varphi$ on $S$. In the Haagerup
construction, $\sigma_t^\varphi(x)=D_\varphi^{it}xD_\varphi^{-it}$
for $x\in S$ and $t\in\mathbb R$ \cite[Section~2, p.~5]{JS}.
The centralizer $S_\varphi$ is the fixed-point algebra of this group
\cite[Theorem~VIII.2.6]{TakesakiII}. Hence, for $v\in\mathcal U(A)\subset S_\varphi$,
\[
 D_\varphi^{it}vD_\varphi^{-it}=\sigma_t^\varphi(v)=v
 \qquad(t\in\mathbb R).
\]
Thus, $v$ commutes with $D_\varphi^{it}=\exp(it\log D_\varphi)$ for every
$t\in\mathbb R$. By the spectral theorem, $v$ commutes with every spectral
projection of $\log D_\varphi$ and hence with
$h=\exp((\log D_\varphi)/p)=D_\varphi^{1/p}$.
Choose increasing finite-dimensional unital subalgebras $S^{(n)}\subset S$
and $A^{(n)}\subset A$ with strongly dense unions. Write $\mu_n$ for
normalized Haar measure on $D_n=\mathcal U(S^{(n)})\times\mathcal U(A^{(n)})$.
Define
\begin{equation}\label{quasi:product-map}
 \begin{gathered}
 P_n=L_\infty(D_n,\mu_n)\overline\otimes G,\qquad
 \tau_n=\left(\int_{D_n}\cdot\,d\mu_n\right)\otimes\tau_G,\\
 \mathcal J_n:L_p(S)\longrightarrow L_p(P_n,\tau_n),\qquad
 (\mathcal J_nx)(g,k)=B(gxk)\quad((g,k)\in D_n),\\
 \|\mathcal J_nx\|_{L_p(P_n,\tau_n)}^p=\int_{D_n}\|B(gxk)\|_{L_p(G,\tau_G)}^p\,d\mu_n(g,k)
             \le\|B\|^p\|x\|_{L_p(S)}^p .
 \end{gathered}
\end{equation}
The function $(g,k)\mapsto B(gxk)$ is continuous in the $L_p(G,\tau_G)$
quasi-norm (by uniform approximation by simple functions, the
integral identity in \eqref{quasi:product-map} also holds for $p<1$).

Fix a free ultrafilter $\omega$ on $\mathbb N$. Let $(Q,\tau_Q)$ be the
tracial ultraproduct of $(P_n,\tau_n)$ \cite[Introduction]{AHUltraproducts},
and let $E=(L_p(P_n,\tau_n))_\omega$ be the quasi-Banach space ultraproduct.
The elements of $E$ are classes $[(x_n)]_\omega$ of sequences
$(x_n)\in\prod_{n\ge1}L_p(P_n,\tau_n)$ with $\sup_n\|x_n\|_p<\infty$,
where $[(x_n)]_\omega=[(y_n)]_\omega$ if and only if
$\lim_{n\to\omega}\|x_n-y_n\|_p=0$.
Every sequence $(b_n)\in\prod_{n\ge1}P_n$ with $\sup_n\|b_n\|_\infty<\infty$
defines an element $[(b_n)]_\omega\in E$.
Write $[b_n]_\omega^Q$ for the image of $(b_n)$ under the quotient
$*$-homomorphism $\ell_\infty(P_n)\to Q$.
Then
$|[b_n]_\omega^Q|^p=\bigl(([b_n]_\omega^Q)^*[b_n]_\omega^Q\bigr)^{p/2}
=[(b_n^*b_n)^{p/2}]_\omega^Q=[|b_n|^p]_\omega^Q$.
By the definitions of the trace $\tau_Q$ and the quasi-norm of $E$, we have
\[
 \tau_Q(|[b_n]_\omega^Q|^p)=\tau_Q([|b_n|^p]_\omega^Q)
 =\lim_{n\to\omega}\tau_n(|b_n|^p)=\|[(b_n)]_\omega\|_E^p .
\]
Thus, $[b_n]_\omega^Q\mapsto[(b_n)]_\omega$
identifies $L_p(Q,\tau_Q)$ isometrically with the closure in $E$ of
the classes of operator-norm bounded sequences
(cf.\ \cite[proof of Lemma~3.6]{HRS}).

For $x=ah$, $a\in S$, we have $(\mathcal J_nx)(g,k)=B((gak)h)$ and
$\|gak\|_\infty\le\|a\|_\infty$. By Lemma~\ref{lem:quasi-equi-maps}
(with \(L=B\) and \(x=h\)), the set \(B(S_1h)\) is \(p\)-equiintegrable.
Apply Lemma~\ref{lem:quasi-spectral-tails} to \(\|a\|_\infty B(S_1h)\),
which contains every \(B((gak)h)\) for \(n\ge1\) and \((g,k)\in D_n\).
Integrating \eqref{quasi:tail-criterion} (with $Y=(\mathcal J_nx)(g,k)=B((gak)h)$)
over $D_n$, we obtain
\begin{equation}\label{quasi:average-tails}
 \lim_{R\to\infty}\sup_n
 \|(\mathcal J_nx)1_{(R,\infty)}(|\mathcal J_nx|)\|_{L_p(P_n,\tau_n)}=0.
\end{equation}
By \eqref{quasi:average-tails}, $[(\mathcal J_nx)]_\omega$ is approximated in $E$ by the classes of
$(\mathcal J_nx)1_{[0,R]}(|\mathcal J_nx|)$ and belongs to $L_p(Q,\tau_Q)$.
Since $Sh$ is dense in $L_p(S)$ by \cite[Lemma~4.7, p.~173]{RX} and
$\sup_n\|\mathcal J_n\|\le\|B\|$, we obtain $[(\mathcal J_n x)]_\omega\in L_p(Q,\tau_Q)$ for every $x\in L_p(S)$.
We have therefore defined a bounded linear map
\begin{equation}\label{quasi:average-map}
 J:L_p(S)\longrightarrow L_p(Q,\tau_Q),\qquad Jx=[(\mathcal J_nx)]_\omega,
 \quad \|Jx\|_{L_p(Q,\tau_Q)}=\lim_{n\to\omega}\|\mathcal J_nx\|_{L_p(P_n,\tau_n)} .
\end{equation}

The unions $\bigcup_n S^{(n)}$ and $\bigcup_n A^{(n)}$ are closed under
exponentiation. By the unitary density argument in
\cite[Theorem~I.7.3, pp.~21--22]{Davidson}, $\bigcup_n D_n$ is strongly
dense in $\mathcal U(S)\times\mathcal U(A)$.
On these groups the strong and strong-$*$ topologies agree and are metrizable
\cite[Remark~I.7.4, p.~22]{Davidson}. Since $D_n\subset D_{n+1}$, for every
$(u,v)\in\mathcal U(S)\times\mathcal U(A)$ we may choose $(u_n,v_n)\in D_n$
with $(u_n,v_n)\to(u,v)$ strong-$*$.
Define $\alpha_n:P_n\to P_n$ and $\alpha:Q\to Q$ by
\begin{equation}\label{quasi:translate}
 (\alpha_n f)(g,k)=f(gu_n,v_nk),\qquad
 \alpha([b_n]_\omega^Q)=[\alpha_n(b_n)]_\omega^Q,
\end{equation}
where $f:D_n=\mathcal U(S^{(n)})\times\mathcal U(A^{(n)})\to G$
represents an element of $P_n$. The first equality holds for $\mu_n$-almost
every $(g,k)\in D_n$. Both maps are trace-preserving $*$-automorphisms.
Indeed, the normalized Haar measure $\mu_n$ on $D_n$ is invariant under
right translation $g\mapsto gu_n$ and left translation
$k\mapsto v_nk$. Hence, $\tau_n\circ\alpha_n=\tau_n$
by the definition of $\tau_n$ in \eqref{quasi:product-map}. In particular, both
$(\alpha_n)_n$ and $(\alpha_n^{-1})_n$ preserve the ideal
$\{(b_n):\sup_n\|b_n\|_\infty<\infty,\ \lim_{n\to\omega}\|b_n\|_2=0\}$.
Write
\[
 \alpha_{u,v}:=\alpha.
\]
Since $\alpha$ preserves $\tau_Q$, it extends to a linear isometry
$(\alpha_{u,v})_p$ on $L_p(Q,\tau_Q)$. Similarly, $(\alpha_n)_p$ denotes
the isometric extension of $\alpha_n$ to $L_p(P_n,\tau_n)$.
We claim that
\begin{equation}\label{quasi:invariance}
 \begin{gathered}
 J(uxv)=(\alpha_{u,v})_p(Jx)
 \quad(x\in L_p(S),\ u\in\mathcal U(S),\ v\in\mathcal U(A)),\\
 \|J(bh)\|_{L_p(Q,\tau_Q)}\ge\delta
 \quad(b\in S_1,\ \|bh\|_{L_p(S)}^p\ge1/4).
 \end{gathered}
\end{equation}
To prove the equality in \eqref{quasi:invariance}, fix $x\in L_p(S)$.
By \eqref{quasi:product-map}, we have
\[
 \|(\alpha_n)_p \mathcal J_nx-\mathcal J_n(uxv)\|_p
 \le \|B\|\,\|u_nxv_n-uxv\|_{L_p(S)}\longrightarrow0 .
\]
Indeed, if $x=w|x|$ is the polar decomposition, write
$x=a_x b_x$ with $a_x=w|x|^{p/2}\in L_2(S)$ and
$b_x=|x|^{1-p/2}\in L_{2p/(2-p)}(S)$.
Then $\|(u_n-u)x\|_p\le\|(u_n-u)a_x\|_2\,\|b_x\|_{2p/(2-p)}\to0$.
Applying this bound to $x^*$
gives $\|x(v_n-v)\|_p\to0$ as well.
The isometry $[(y_n)]_\omega\mapsto[((\alpha_n)_p y_n)]_\omega$ of $E$
agrees with $\alpha_{u,v}$ on $Q$ by \eqref{quasi:translate}, hence with
$(\alpha_{u,v})_p$ on the closure $L_p(Q,\tau_Q)$ of $Q$ in $E$.
By \eqref{quasi:average-map}, for every $x\in L_p(S)$ we have
\[
 \begin{aligned}
 \|J(uxv)-(\alpha_{u,v})_p(Jx)\|_{L_p(Q,\tau_Q)}
 &=\lim_{n\to\omega}
   \|\mathcal J_n(uxv)-(\alpha_n)_p\mathcal J_nx\|_{L_p(P_n,\tau_n)}\\
 &\le \|B\|\lim_{n\to\omega}\|uxv-u_nxv_n\|_{L_p(S)}=0 .
 \end{aligned}
\]

On the other hand, for $b\in S_1$ and $(g,k)\in D_n$, we have
\[
 gbhk=(gbk)h,\qquad \|gbk\|_\infty\le1,\qquad \|(gbk)h\|_p=\|bh\|_p.
\]
Thus, if $\|bh\|_p^p\ge1/4$,
then $\|B(gbhk)\|_p\ge\delta$ by \eqref{quasi:finite-orbit}.
We have
\[
 \begin{aligned}
 \|J(bh)\|_{L_p(Q,\tau_Q)}^p
 &\overset{\text{\scriptsize\eqref{quasi:average-map}}}{=}
   \lim_{n\to\omega}\|\mathcal J_n(bh)\|_{L_p(P_n,\tau_n)}^p\\
 &\overset{\text{\scriptsize\eqref{quasi:product-map}}}{=}
   \lim_{n\to\omega}\int_{D_n}
   \|B(gbhk)\|_{L_p(G,\tau_G)}^p\,d\mu_n(g,k)\\
 &\ge\lim_{n\to\omega}\delta^p\mu_n(D_n)=\delta^p .
 \end{aligned}
\]
This proves the claim.

To prove the third assertion, apply \eqref{quasi:one-sided} with \(L=J\)
and \(x=h\). It follows that \(\{J(hb):b\in S_1\}\) is \(p\)-equiintegrable.
Fix $u\in\mathcal U(S)$, $b\in S_1$ and $f\in\operatorname{Proj}(Q)$,
and write $\alpha=\alpha_{u,1}$. Taking $x=hb$ and $v=1$ in
\eqref{quasi:invariance} gives $J(uhb)=\alpha_p(J(hb))$.
For $a\in Q$ and $y\in L_p(Q,\tau_Q)$, we have
$\alpha_p(aya)=\alpha(a)\alpha_p(y)\alpha(a)$. Taking
$a=\alpha^{-1}(f)$ and using the isometry of $\alpha_p$, we obtain
\[
 \begin{aligned}
 \|fJ(uhb)f\|_{L_p(Q,\tau_Q)}
 &=\|f\,\alpha_p(J(hb))\,f\|_{L_p(Q,\tau_Q)}\\
 &=\|\alpha_p(\alpha^{-1}(f)J(hb)\alpha^{-1}(f))\|_{L_p(Q,\tau_Q)}\\
 &=\|\alpha^{-1}(f)J(hb)\alpha^{-1}(f)\|_{L_p(Q,\tau_Q)} .
 \end{aligned}
\]
Since $\tau_Q\circ\alpha=\tau_Q$, we also have
$\tau_Q(\alpha^{-1}(f))=(\tau_Q\circ\alpha)(\alpha^{-1}(f))=\tau_Q(f)$.
For every real $t>0$, the projection $\alpha_{u,1}^{-1}(f)$ therefore has
trace at most $t$ whenever $\tau_Q(f)\le t$. Since $\{J(hb):b\in S_1\}$
is $p$-equiintegrable, \cite[Corollary~4.9(iv), p.~176]{RX} yields
\[
 \sup_{\substack{f\in\operatorname{Proj}(Q),\ \tau_Q(f)\le t\\
                 u\in\mathcal U(S),\ b\in S_1}}
       \|fJ(uhb)f\|_{L_p(Q,\tau_Q)}
 \le
 \sup_{\substack{e\in\operatorname{Proj}(Q),\ \tau_Q(e)\le t\\ b\in S_1}}
       \|eJ(hb)e\|_{L_p(Q,\tau_Q)}
 \xrightarrow[t\downarrow0]{}0 .
\]
For a self-adjoint contraction $c$, $w_c=c+i(1-c^2)^{1/2}$ is unitary
and $c=(w_c+w_c^*)/2$. For $a\in S_1$, we have
$2a=w_{\Rea a}+w_{\Rea a}^*+i w_{\Ima a}+i w_{\Ima a}^*$.
For $b\in S_1$ and $f\in\operatorname{Proj}(Q)$, we obtain
\[
\begin{aligned}
 \|fJ(ahb)f\|_p^r
 &=2^{-r}\bigl\|fJ\bigl((w_{\Rea a}+w_{\Rea a}^*
                +i w_{\Ima a}+i w_{\Ima a}^*)hb\bigr)f\bigr\|_p^r\\
 &\le 2^{-r}\bigl(\|fJ(w_{\Rea a}hb)f\|_p^r
                  +\|fJ(w_{\Rea a}^*hb)f\|_p^r\\
 &\hspace{4.8em} {}+\|fJ(w_{\Ima a}hb)f\|_p^r
                  +\|fJ(w_{\Ima a}^*hb)f\|_p^r\bigr)\\
 &\le 4\cdot 2^{-r}\sup_{u\in\mathcal U(S)}\|fJ(uhb)f\|_p^r .
\end{aligned}
\]
Thus,
\begin{equation}\label{quasi:double-ui}
 J(S_1hS_1)=\{J(ahb):a,b\in S_1\}
 \quad\text{is $p$-equiintegrable in $L_p(Q,\tau_Q)$.}\qedhere
\end{equation}
\end{proof}

\subsection{Center-valued traces and equiintegrability}

Let \(A\subset S_\varphi\), \((Q,\tau_Q)\) and \(J:L_p(S)\to L_p(Q,\tau_Q)\)
be as in Proposition~\ref{prop:quasi-average}. The restriction $\varphi|_A$ is a faithful
normal tracial state. For the normalized center-valued trace $T_A$ of
$A$ and the state $\nu=\varphi|_{Z(A)}$, we have
$\varphi(a)=\nu(T_A(a))$ for $a\in A$.

For $Y\in L_p(Q,\tau_Q)$ and $t>0$, define the truncated second moment
\begin{equation}\label{quasi:trunc-def}
 H_t(Y)=\tau_Q\bigl(\min\{|Y|^2,t^2\}\bigr),\qquad
 H_t(J(uxv))=H_t(Jx)
 \quad(u\in\mathcal U(S),\ v\in\mathcal U(A),\ x\in L_p(S)).
\end{equation}
The equality follows from \eqref{quasi:invariance}, since
\(\alpha_{u,v}\) preserves \(\tau_Q\) by Proposition~\ref{prop:quasi-average}.
For any projection
$e\in Q$, split the trace of $D_t=\min\{|Y|^2,t^2\}$ over $e$ and $1-e$.
Using $D_t\le t^2$ on the first compression and $D_t\le |Y|^2$ on the
second, we obtain
\begin{equation}\label{quasi:trunc-projection}
 H_t(Y)\le t^2\tau_Q(e)+\|Y(1-e)\|_{L_2(Q,\tau_Q)}^2,
 \qquad\text{with equality for }e=1_{(t,\infty)}(|Y|).
\end{equation}
For $Y_1,\ldots,Y_k\in L_p(Q,\tau_Q)$ take
$e_i=1_{(t,\infty)}(|Y_i|)$ and $e=\bigvee_{i=1}^ke_i$. Then we have
$\tau_Q(e)\le\sum_i\tau_Q(e_i)$ and
$\|Y_i(1-e)\|_2\le\|Y_i(1-e_i)\|_2$. For independent symmetric
signs $\varepsilon_i$, orthogonality in the sign average gives
\begin{equation}\label{quasi:trunc-signs}
 \mathbb E H_t\!\left(\sum_{i=1}^k\varepsilon_iY_i\right)
 \le t^2\tau_Q(e)+\sum_{i=1}^k\|Y_i(1-e)\|_2^2
 \le\sum_{i=1}^k H_t(Y_i).
\end{equation}
For $k=2$, take $e=e_1\vee e_2$ in \eqref{quasi:trunc-projection}. Then
$H_t((Y_1+Y_2)/2)\le H_t(Y_1)+H_t(Y_2)$, since
$\|(Y_1+Y_2)(1-e)/2\|_2^2\le\frac12\sum_{i=1}^2\|Y_i(1-e)\|_2^2$.
Given $0\le a\le1$ in $S$, let $w=a+i(1-a^2)^{1/2}\in\mathcal U(S)$.
Taking $Y_1=J(wx)$ and $Y_2=J(w^*x)$, we deduce from \eqref{quasi:trunc-def} that
\begin{equation}\label{quasi:positive-multiplier}
 H_t(J(ax))\le 2H_t(Jx)\qquad(x\in L_p(S),\ t>0).
\end{equation}

\begin{lemma}\label{lem:small-projections}
Suppose that $F_n\in\operatorname{Proj}(A)$, $d_n=T_A(F_n)$,
$z_n=s(d_n)$, and $m_n$ are positive integers tending to infinity. If
\begin{equation}\label{quasi:small-projections}
 d_n\le m_n^{-1}z_n,\qquad \nu(z_n)\ge1/3,\qquad
 F_nh=d_n^{1/p}x_n
\end{equation}
for $x_n\in L_p(S)$, then $(Jx_n)_{n\ge1}$ is not $p$-equiintegrable.
\end{lemma}

\begin{proof}
Fix $n$ and write $F=F_n$, $d=d_n$, $z=z_n$, $m=m_n$. We decompose
$z$ into orthogonal projections, each equivalent in $A$ to a subprojection
of $F$. Starting with $g_0=z$, choose $e_j\le g_{j-1}$ and $F_j'\le F$
with
\begin{equation}\label{quasi:tile}
 e_j\sim F_j',\qquad T_A(e_j)=\min\{T_A(g_{j-1}),d\},\qquad
 g_j=g_{j-1}-e_j.
\end{equation}
Indeed, consider the central projection $b=1_{(-\infty,0]}(T_A(g_{j-1})-d)\in Z(A)$.
Projection comparison in $A$ gives $bg_{j-1}\precsim bF$ and
$(1-b)F\precsim(1-b)g_{j-1}$ \cite[III.2.5]{Blackadar}.
Taking the smaller projection on each central summand and an equivalent
subprojection of the larger produces the projections $e_j,F_j'$ in \eqref{quasi:tile}.
By \eqref{quasi:tile}, we have $T_A(g_j)=(z-jd)_+\downarrow0$ and
$\varphi(g_j)=\nu(T_A(g_j))\to0$. Choose $k_n$ so that
$W_n:=\sum_{j=1}^{k_n}e_j=z-g_{k_n}$ satisfies
$\varphi(W_n)\ge1/4$, which is possible since $\varphi(z)=\nu(z)\ge1/3$.

For $j\ge2$, the central support of $e_j$ lies under
$s(T_A(g_{j-1}))=s((z-(j-1)d)_+)$. On that support, we have $d\le1/(j-1)\le2/j$.
The bound $d\le2/j$ also holds for $j=1$, because $d\le1$. Combining
$d\le2/j$ with $d\le1/m$, we obtain $d\le3/(m+j)$ on the central support of $e_j$,
which is also the central support of $F_j'$. Hence, for
$c_{m,j}=(3/(m+j))^{1/p}$, the operator $a_j=c_{m,j}^{-1}d^{1/p}F_j'$
is a positive contraction. Choose $u_j\in\mathcal U(A)$ such that
$e_j=u_j F_j' u_j^*$, using finiteness of $A$ and $e_j\sim F_j'$.
Since $u_j$ commutes with $h$ and $F_j' h=F_j' d^{1/p}x_n$ by
\eqref{quasi:small-projections}, we obtain
\begin{equation}\label{quasi:tile-vector}
 e_jh=c_{m,j}u_j(a_jx_n)u_j^*,\qquad 1\le j\le k_n .
\end{equation}

Write $Y_n=Jx_n$. Left multiplication by the unitary
$1_S+\sum_{j=1}^{k_n}(\varepsilon_j-1)e_j$ sends $W_nh$ to
$\sum_{j=1}^{k_n}\varepsilon_j e_jh$ for each choice of signs.
By \eqref{quasi:trunc-def} and \eqref{quasi:trunc-signs}, we obtain
\begin{equation}\label{quasi:tile-sum}
 H_t(J(W_nh))\le\sum_{j=1}^{k_n} H_t(J(e_jh))
 \le 2\sum_{j=1}^{k_n}H_t(c_{m,j}Y_n)\qquad(t>0).
\end{equation}
The second inequality uses \eqref{quasi:tile-vector}, \eqref{quasi:trunc-def} and \eqref{quasi:positive-multiplier}
applied to $a_j(c_{m,j}x_n)$.

Suppose $(Y_n)$ were $p$-equiintegrable. For fixed $t>0$ and $s\ge0$, define
\[
 G_{m,t}(s)=\int_m^\infty \min\{(3/u)^{2/p}s^2,t^2\}\,du .
\]
Since the integrand decreases in $u$, we have
$\sum_{j\ge1}\min\{c_{m,j}^2s^2,t^2\}\le G_{m,t}(s)$.
Splitting the integral at $u=3(s/t)^p$ gives
$G_{0,t}(s)=6t^{2-p}s^p/(2-p)$, and $G_{m,t}(R)\downarrow0$ for every fixed
$R>0$. For every $R>0$, we apply the spectral theorem to $|Y_n|$ in \eqref{quasi:tile-sum} to obtain
\begin{equation}\label{quasi:moment-upper}
 H_t(J(W_nh))\le 2\tau_Q(G_{m_n,t}(|Y_n|))
 \le 2G_{m_n,t}(R)+\frac{12t^{2-p}}{2-p}
                  \big\|Y_n1_{(R,\infty)}(|Y_n|)\big\|_p^p .
\end{equation}
Using \eqref{quasi:tail-criterion} for $(Y_n)_{n\ge1}$ in
\eqref{quasi:moment-upper}, first take $\limsup_{n\to\infty}$ and then let
$R\to\infty$. We obtain
\begin{equation}\label{quasi:moment-zero}
 \lim_{n\to\infty}H_t(J(W_nh))=0\qquad(t>0).
\end{equation}

On the other hand, $W_n$ commutes with $h$ and
$\|W_nh\|_p^p=\varphi(W_n)\ge1/4$. Thus, $Z_n=J(W_nh)$ satisfies
$\|Z_n\|_p\ge\delta$ by \eqref{quasi:invariance}. The family
$(Z_n)\subset J(S_1h)$ is $p$-equiintegrable by \eqref{quasi:one-sided}.
By \eqref{quasi:tail-criterion}, choose $t>0$ with
$\sup_n\|Z_n1_{(t,\infty)}(|Z_n|)\|_p^p\le\delta^p/2$.
Since $\tau_Q$ is a state and $p<2$, we have
\begin{equation}\label{quasi:moment-lower}
 H_t(Z_n)
 \ge\|Z_n1_{[0,t]}(|Z_n|)\|_2^2
 \ge\|Z_n1_{[0,t]}(|Z_n|)\|_p^2
 \ge(\delta^p/2)^{2/p},
\end{equation}
contradicting \eqref{quasi:moment-zero} for this $t$ and $Z_n=J(W_nh)$.
\end{proof}

\subsection{Application to the discrete decomposition}

\begin{lemma}\label{lem:quasi-discrete-setup}
Let $M$ be a nonzero hyperfinite type III von Neumann algebra with
separable predual. There exists a hyperfinite type III von Neumann algebra
$M_0=B\rtimes_\alpha\mathbb Z$ with separable predual, where $B$ has type
$\mathrm{II}_\infty$ and $\alpha$ is an automorphism of $B$. There are a
faithful normal semifinite trace $\tau$ on $B$, a constant $0<\lambda<1$,
and a projection $P\in B$ of central support $1_B$ such that
\begin{equation}\label{quasi:discrete-contraction}
 \tau(\alpha(b))\le\lambda\tau(b)\quad(b\in B_+),\qquad \tau(P)=1.
\end{equation}
Let $E_B:M_0\to B$ be the canonical conditional expectation and set
\begin{equation}\label{quasi:corner-state}
 S=PM_0P,\qquad A=PBP,\qquad \varphi=(\tau\circ E_B)|_S.
\end{equation}
Then $S$ is hyperfinite of type III, $\varphi$ is a faithful normal state
on $S$, and $A$ is a unital finite hyperfinite type II subalgebra of
$S_\varphi$. For every $0<p<\infty$ there is a linear isometry
$\iota_p:L_p(S)\to L_p(M)$.
\end{lemma}

\begin{proof}
Write $M=\int_X^\oplus M_x\,d\mu(x)$ for the central decomposition over a
standard probability space. Almost every factor $M_x$ is hyperfinite of type
$\mathrm{III}_{\theta(x)}$ for some $\theta(x)\in[0,1]$. The type
parameter $x\mapsto\theta(x)$ is measurable
\cite[Theorem~1.1 and its proof in Section~4]{Golodets1975}.
At least one of the sets
$\{\theta=0\}$, $\{0<\theta<1\}$ and $\{\theta=1\}$ has positive
measure. Restrict $M$ to the corresponding central summand $zM$.
The identification $L_p(zM)=zL_p(M)z$ is isometric
\cite[Fact~1.4, p.~155]{RX}.

If $\theta=1$ on this summand, then $M_x\cong R_\infty$ almost
everywhere by \cite[Corollary~2.4]{Haagerup1987}. The trivialization
theorem for a measurable field of mutually isomorphic factors
\cite[Theorem~3, p.~437]{TakesakiReduction1969} gives
$zM\cong Z(zM)\overline\otimes R_\infty$.
Since $\log(1/2)/\log(1/3)$ is irrational,
$R_\infty\cong R_{1/2}\overline\otimes R_{1/3}$
\cite[Example~1.6]{Haagerup1987}. Under the resulting identification
\[
 zM\cong R_{1/2}\overline\otimes
            \bigl(Z(zM)\overline\otimes R_{1/3}\bigr),
\]
a faithful normal state $\psi$ on the second tensor factor gives the
conditional expectation $\mathrm{id}_{R_{1/2}}\otimes\psi$ onto $R_{1/2}$.
Equation \eqref{eq:expected-isometry} yields a linear isometry
$L_p(R_{1/2})\to L_p(zM)$ for every $0<p<\infty$.
We can therefore use the construction for $0<\theta<1$ below with
$R_{1/2}$ in place of $zM$.

Suppose that $0<\theta<1$ almost everywhere on the chosen
summand. Since $\{0<\theta<1\}=\bigcup_{k\ge2}\{0<\theta\le 1-1/k\}$,
there is a further nonzero central summand $M_0$ on which
$\theta\le\lambda:=1-1/k<1$ for some $k\ge2$.
By \cite[Theorem~5.3, p.~277]{Golodets1975},
$M_0=B\rtimes_\alpha\mathbb Z$ with $B$ of type $\mathrm{II}_\infty$.
In Golodets' decomposition over the type parameter $s$, the trace
$\tau=\int^\oplus \tau_s\,d\eta(s)$ and the automorphism
$\alpha=\int^\oplus\alpha_s\,d\eta(s)$ satisfy $\tau_s\circ\alpha_s=s\tau_s$
almost everywhere. Thus, for every $b\in B_+$,
\[
 \tau(\alpha(b))=\int s\,\tau_s(b_s)\,d\eta(s)\le\lambda\tau(b).
\]

Suppose now that $\theta=0$ almost everywhere on $zM$.
The discrete decomposition \cite[Theorem~2(b),(c), p.~62]{Korol1991}
gives $zM=B_0\rtimes_{\alpha_0}\mathbb Z$, with $B_0$ of type
$\mathrm{II}_\infty$, $Z(zM)=Z(B_0)^{\alpha_0}$, and a faithful
normal semifinite $Z(zM)$-valued trace $\Phi$ satisfying
$\Phi\circ\alpha_0\le\lambda_0\Phi$. Here $\lambda_0\in Z(zM)$ and
$1_{(0,1)}(\lambda_0)=z$. Choose $k\ge2$ such that
$e:=1_{(0,\,1-1/k]}(\lambda_0)\ne0$, and set $\lambda=1-1/k$.
Since $\alpha_0(e)=e$, we have
$M_0:=eM=B\rtimes_\alpha\mathbb Z$ for $B=eB_0$ and
$\alpha=\alpha_0|_B$. For a faithful normal state $\psi_0$ on
$Z(M_0)$, the composition $\tau=\psi_0\circ\Phi|_{B_+}$ is a faithful
normal semifinite trace \cite[Proposition~7.7]{Hiai2021}, and
\[
 \tau(\alpha(b))\le\psi_0(e\lambda_0\Phi(b))\le\lambda\tau(b)
 \qquad(b\in B_+).
\]

In all cases, $M_0$ is hyperfinite of type III, and the preceding
isometries give $L_p(M_0)\hookrightarrow L_p(M)$ isometrically.
To choose $P$, write $c_B(q)$ for the central support of a projection
$q\in B$. By semifiniteness of $\tau$ and countable decomposability of
$B$, choose projections $q_j\in B$ of finite trace such that
$\bigvee_j c_B(q_j)=1_B$. The central projections
$z_j=c_B(q_j)(1-\bigvee_{i<j}c_B(q_i))$ are orthogonal with sum $1_B$.
Each nonzero $z_jq_j$ can be divided into arbitrarily many equivalent
projections in its type II corner \cite[III.2.5.4(iii)]{Blackadar}.
Choose one such projection $p_j$ with $\tau(p_j)\le2^{-j}$.
Equivalence of the pieces gives $c_B(p_j)=c_B(z_jq_j)=z_j$.
Set $p_j=0$ if $z_j=0$. Then $P=\sum_{j\ge1}p_j$ has central
support $1_B$ and $0<\tau(P)\le1$. Replacing $\tau$ by
$\tau/\tau(P)$ gives \eqref{quasi:discrete-contraction}.

The corner $S=PM_0P$ is hyperfinite of type III. The expected
subalgebra $B\subset M_0$ is hyperfinite, as is its finite type II
corner $A=PBP$ \cite[IV.2.1.2 and Corollary~IV.2.6.2]{Blackadar}.
Since $E_B(P)=P$ and $\tau(P)=1$, the restriction
$\varphi=(\tau\circ E_B)|_S$ is a faithful normal state. For $a\in A$
and $x\in S$, bimodularity of $E_B$ gives
\[
 \varphi(ax)=\tau(aE_B(x))=\tau(E_B(x)a)=\varphi(xa),
\]
so $A\subset S_\varphi$. Finally,
$L_p(S)=P L_p(M_0)P$ isometrically \cite[Fact~1.4, p.~155]{RX}.
Composing this inclusion with the isometry $L_p(M_0)\to L_p(M)$
constructed above yields $\iota_p$.
\end{proof}

The following construction gives a projection $F_n$ whose center-valued trace
is small on a set of central measure at least $1/3$. It also relates
$F_nh$ to a two-sided multiple of $h$.

\begin{lemma}\label{lem:discrete-small-corners}
Let $B$ be a semifinite von Neumann algebra with separable predual, let $\alpha$ be a normal
$*$-automorphism of $B$, and suppose that $B\rtimes_\alpha\mathbb Z$
is of type III. Let $\tau$ be a faithful normal semifinite trace on $B$
with $\tau\circ\alpha\le\lambda\tau$ for some $0<\lambda<1$.
Choose a projection $P\in B$ with full central support and $\tau(P)=1$.
Set $S=P(B\rtimes_\alpha\mathbb Z)P$, $A=PBP$ and
$\varphi=(\tau\circ E_B)|_S$, where $E_B$ is the canonical expectation.
Let $T_A$ be the normalized center-valued trace on $A$ and set
$\nu=\varphi|_{Z(A)}$.

For each $n\ge1$ there is a partial isometry $v_n\in S$ with
$q_n=v_n^*v_n\in A$ and $F_n=v_nv_n^*\in A$. There are
$t_n,r_n\in Z(A)_+$ such that, writing $d_n=T_A(F_n)$, we have
\begin{equation}\label{quasi:array-data}
 \begin{gathered}
 d_n=t_nr_n\le z_n/n,\qquad z_n=s(d_n)=s(t_n)=s(r_n),\\
 z_n/2\le r_n\le2z_n,\qquad \nu(z_n)\ge1/3.
 \end{gathered}
\end{equation}
For every $p>0$, we have
\begin{equation}\label{quasi:array-modular}
 F_n h=t_n^{1/p}v_n h v_n^*,\qquad h=D_\varphi^{1/p}.
\end{equation}
\end{lemma}

\begin{proof}
Since the central support of $P$ is $1$, the map $z\mapsto zP$ identifies
$Z(B)$ with $Z(A)$ and $\nu(zP)=\tau(zP)$. Represent $Z(B)$ as
$L_\infty(Y,\nu)$ for a standard probability space $(Y,\nu)$, and write
$\alpha(g)=g\circ T^{-1}$ for the associated invertible nonsingular
transformation $T$ of $Y$. Thus $\tau(gP)=\int_Y g\,d\nu$.
The transformation $T$ is conservative. Indeed, if $W$ were a wandering
set of positive measure, then $w=1_W\in Z(B)$ would satisfy
$w\alpha^k(w)=0$ for $k\ne0$. Let $u$ be the implementing unitary, so $ubu^*=\alpha(b)$ for $b\in B$.
For every $b\in B$, we would have $wbu^kw=bw\alpha^k(w)u^k=0$ for $k\ne0$.
The ultraweak density of the finite Fourier sums would then give
$w(B\rtimes_\alpha\mathbb Z)w=wBw$, a nonzero semifinite corner of a
type III algebra.

Fix $n$ and choose $L\ge1$ with $2\lambda^L\le1/n$. Write
$\rho_k=d(T^k_*\nu)/d\nu$ for $k\ge1$, so that
$\int_H\rho_k\,d\nu=\nu(T^{-k}H)$ for measurable $H\subset Y$.
We first construct measurable sets $E_l$ and integers $k_l\ge L$ such that
the $E_l$ are pairwise disjoint, the sets $H_l=T^{k_l}E_l$ are pairwise
disjoint, and
\begin{equation}\label{quasi:return-sets}
 \tfrac12\le\rho_{k_l}\le2 \ \text{on }H_l,\qquad
 \nu\Bigl(\bigcup_l H_l\Bigr)\ge\tfrac13.
\end{equation}
The Maharam extension
$\widetilde T(y,s)=(Ty,s\rho_1(Ty))$ on
$(Y\times(0,\infty),\nu\otimes ds)$ is conservative by
Maharam's theorem \cite{Maharam1964}. The Radon--Nikodym chain rule gives
$\widetilde T^k(y,s)=(T^ky,s\rho_k(T^ky))$.
For any measurable $U\subset Y$ of positive measure, conservativity gives
a return of $U\times[1,2]$ to itself of positive product measure at
some time $k\ge L$. At every such return both $s$ and $s\rho_k(T^ky)$
lie in $[1,2]$. Consequently,
\[
 E=U\cap T^{-k}U\cap\{y:\tfrac12\le\rho_k(T^ky)\le2\}
\]
has positive measure, and $E,T^kE\subset U$.
Starting with $U=Y$, choose such a pair $(E,T^kE)$ for which
$\nu(E\cup T^kE)$ is at least half the supremum over all admissible
pairs in $U$. Remove $E\cup T^kE$ and repeat on the remainder.
If the uncovered set had positive measure after countably many steps,
it would contain an admissible pair whose union has measure $\eta>0$.
At each step we would then have removed measure at least $\eta/2$,
contrary to $\nu(Y)=1$. The chosen pairs $(E_l,H_l)$ therefore cover $Y$
up to a null set. Since $\nu(E_l)=\int_{H_l}\rho_{k_l}\,d\nu\le2\nu(H_l)$,
we obtain $1\le\sum_l(\nu(E_l)+\nu(H_l))\le3\nu(\bigcup_l H_l)$.
This proves \eqref{quasi:return-sets}.

For each $k\ge1$, the Radon--Nikodym theorem for traces gives
$b_k\in Z(B)_+$ with $\tau(\alpha^k(a))=\tau(b_ka)$ for $a\in B_+$
\cite[Theorems~7.25 and~7.22(1)]{Hiai2021}.
Since $\tau\circ\alpha^k$ is faithful and dominated by $\lambda^k\tau$,
we have $s(b_k)=1$ and $b_k\le\lambda^k1$. Set $\beta_k=\alpha^k(b_k)$.
On each $H_l$ define $t_n=\beta_{k_l}$ and $r_n=\rho_{k_l}$,
and set both functions equal to zero outside $H=\bigcup_l H_l$.
For $Q_l=1_{E_l}P$ and $g\in Z(B)_+$, a change of variables gives
\begin{equation}\label{quasi:return-trace}
 \begin{aligned}
 \tau(g\alpha^{k_l}(Q_l))
 &=\tau\bigl(b_{k_l}\alpha^{-k_l}(g)Q_l\bigr)
   =\int_{E_l} b_{k_l}(y)g(T^{k_l}y)\,d\nu(y)\\
 &=\int_{H_l}g\,t_nr_n\,d\nu
 \le \tfrac1n\,\tau(g1_{H_l}P).
 \end{aligned}
\end{equation}
Here $0<t_nr_n\le2\lambda^{k_l}\le1/n$ on $H_l$.

By projection comparison in $B$ \cite[III.1.1.10]{Blackadar},
\eqref{quasi:return-trace} implies
$\alpha^{k_l}(Q_l)\precsim 1_{H_l}P$.
Indeed, both projections have finite trace. On any central summand where
$1_{H_l}P$ is equivalent to a subprojection of $\alpha^{k_l}(Q_l)$,
\eqref{quasi:return-trace} applied to that central projection forces
the complementary subprojection to have trace zero.
Choose $a_l\in B$ with $a_l^*a_l=\alpha^{k_l}(Q_l)$ and
$f_l:=a_la_l^*\le1_{H_l}P$. Then $w_l=a_lu^{k_l}Q_l$ has initial
projection $Q_l$ and final projection $f_l$. Both belong to $A$,
so $w_l=Pw_lP\in S$. The $Q_l$ are pairwise orthogonal,
as are the $f_l$. Hence $v_n=\sum_l w_l$ converges strongly to
a partial isometry with $q_n=\sum_l Q_l$ and $F_n=\sum_l f_l$.
For $g\in Z(B)_+$, traciality and \eqref{quasi:return-trace} give
\[
 \int_Y g\,T_A(F_n)\,d\nu
 =\tau(gF_n)=\sum_l\tau(g\alpha^{k_l}(Q_l))
 =\int_Y g\,t_nr_n\,d\nu .
\]
Thus, $T_A(F_n)=t_nr_n$. The supports of $t_n$ and $r_n$ both equal
$z_n=1_HP$ under the identification of the centers.
Equation \eqref{quasi:return-sets} proves the remaining assertions
of \eqref{quasi:array-data}.

By the dual-weight formula
\cite[Lemma~X.1.13 and Theorem~X.1.17]{TakesakiII},
the modular group of $\tau\circ E_B$ fixes $B$ pointwise and satisfies
\[
 \sigma_s^{\tau\circ E_B}(u^k)
 =u^k (D(\tau\circ\alpha^k):D\tau)_s
 =u^k b_k^{is}=\beta_k^{is}u^k
 \qquad(k\ge1,\ s\in\mathbb R).
\]
Since $P$ is fixed,
the restriction of this group to $S$ is $\sigma^\varphi$.
Thus, $\sigma_s^\varphi(w_l)=\beta_{k_l}^{is}w_l$.
Summing over the disjoint ranges $f_l\le1_{H_l}P$ yields
\[
 D_\varphi^{is}v_nD_\varphi^{-is}=t_n^{is}v_n\qquad(s\in\mathbb R),
\]
where the imaginary powers of $t_n$ are taken on $z_n$.
The projections $q_n,F_n$ and the operator $t_n$ belong to
$A\subset S_\varphi$, so they commute with the spectral projections of
$D_\varphi$. On the range of $F_n$, it follows that
\[
 v_n(D_\varphi|_{q_n})^{is}v_n^*
 =(t_n|_{F_n})^{-is}(D_\varphi|_{F_n})^{is}
 =\bigl((t_n|_{F_n})^{-1}(D_\varphi|_{F_n})\bigr)^{is}
 \qquad(s\in\mathbb R).
\]
Here the restrictions act on the ranges of the indicated projections.
Since $F_n\le s(t_n)$, the restriction of $t_n$ to $F_n$ is injective,
and the last product is defined by the joint spectral calculus.
The uniqueness of the generator of a unitary group identifies
the logarithms of the positive operators on the two sides.
Applying the function $\xi\mapsto e^{\xi/p}$ to those logarithms gives
$v_nD_\varphi^{1/p}v_n^*=t_n^{-1/p}D_\varphi^{1/p}F_n$ on $F_n$.
Multiplication by $t_n^{1/p}$ in the same functional calculus proves
\eqref{quasi:array-modular}.
\end{proof}

\begin{proof}[Proof of Theorem~\ref{thm:interior}]
Suppose that $T:L_p(M)\hookrightarrow L_p(F\overline\otimes B(\ell_2))$.
Let $S,A,\varphi$ and $\iota_p:L_p(S)\to L_p(M)$ be as in
Lemma~\ref{lem:quasi-discrete-setup}, and set $h=D_\varphi^{1/p}$.
Apply Proposition~\ref{prop:quasi-finite-target} to $T\iota_p$ and then
Proposition~\ref{prop:quasi-average} to the resulting map $B$ and the
inclusion $A\subset S_\varphi$. We obtain a map
$J:L_p(S)\to L_p(Q,\tau_Q)$ satisfying \eqref{quasi:invariance}
and \eqref{quasi:double-ui}.

For each $n\ge1$, take $v_n,F_n,t_n,r_n$ from
Lemma~\ref{lem:discrete-small-corners}. Write $d_n=T_A(F_n)$,
$z_n=s(d_n)$ and $\nu=\varphi|_{Z(A)}$. By \eqref{quasi:array-data},
$d_n=t_nr_n\le z_n/n$, $\nu(z_n)\ge1/3$ and $r_n\ge z_n/2$.
Define $r_n^{-1/p}$ by functional calculus on $z_n$ and extend it
by zero on $1_A-z_n$. Then $\|r_n^{-1/p}\|_\infty\le 2^{1/p}$.
Since $z_nF_n=F_n$ and $F_n=v_nv_n^*$, we have $z_nv_n=v_n$.
For $x_n=r_n^{-1/p}v_n h v_n^*$, \eqref{quasi:array-modular} gives
\[
 d_n^{1/p}x_n=t_n^{1/p}z_nv_n h v_n^*=F_nh .
\]
Thus, \eqref{quasi:small-projections} holds with $m_n=n$, and
Lemma~\ref{lem:small-projections} implies that $(Jx_n)_{n\ge1}$
is not $p$-equiintegrable. On the other hand,
$2^{-1/p}r_n^{-1/p}v_n$ and $v_n^*$ are contractions in $S$, so
\[
 Jx_n=2^{1/p}J\bigl((2^{-1/p}r_n^{-1/p}v_n)h v_n^*\bigr)
 \in 2^{1/p}J(S_1hS_1)\qquad(n\ge1).
\]
This contradicts the $p$-equiintegrability of $J(S_1hS_1)$
in \eqref{quasi:double-ui}.
\end{proof}

\section{The proof of Theorem~\ref{thm:unified}}\label{sec:reduction}


\subsection{Expected hyperfinite subalgebras}
\label{sec:expected-transfer}

An inclusion $P\subset M$ is said to be \emph{with expectation}
if there is a faithful normal conditional expectation $E:M\to P$.
{For an inclusion $P\subset M$ with expectation $E:M\to P$,
if $P$ is $\sigma$-finite, choose a faithful normal state $\psi$ on $P$
and put $\varphi:=\psi\circ E$, a faithful normal state on $M$.
Since $\varphi\circ E=\varphi$, the algebra $P$ is invariant under
$\sigma^\varphi$ and $\sigma^\varphi_t|_P=\sigma^\psi_t$ for $t\in\mathbb R$.
By \cite[Proposition~3.1]{JungeFubini}, the inclusion $P\subset M$
therefore induces a linear isometry}
\begin{equation}\label{eq:expected-isometry}
 j_p:L_p(P,\psi)\longrightarrow L_p(M,\varphi)\qquad(0<p<\infty).
\end{equation}
For $1\le p<\infty$, the expectation $E$ also induces a contractive
left inverse $E_p$ for $j_p$, namely
\begin{equation}\label{eq:expected-lp}
 j_p:L_p(P)\longrightarrow L_p(M),\qquad
 E_p:L_p(M)\longrightarrow L_p(P),\qquad E_pj_p=I,
\end{equation}
where $j_p$ is isometric and $E_p$ is contractive,
$1\le p<\infty$; see~\cite[Remark~1.1 and Example~5.8]{HJX}.

The following result is certainly known to experts. Due to the lack of suitable references, we include a full proof below. 
\begin{proposition}\label{prop:expected-typeiii-general}
{Every nonzero type III von Neumann algebra $M$ with
separable predual contains a unital hyperfinite type III von Neumann
subalgebra $P$ with expectation and with separable predual.}
\end{proposition}

\begin{proof}
For a type III factor with separable predual, an irreducible hyperfinite
type III subfactor with normal conditional expectation exists by
\cite[Corollary~3.4]{PopaExpected} in type $\mathrm{III}_0$
and \cite[Theorem~D(i)]{Marrakchi20} in type $\mathrm{III}_1$.
For a type $\mathrm{III}_\lambda$ factor $M$, $0<\lambda<1$, choose a
faithful normal state $\varphi$ on $M$ whose modular group
$(\sigma_s^\varphi)_{s\in\mathbb R}$ has period $2\pi/|\log\lambda|$
\cite[Remark~1.3.3 and Theorem~4.3.2(a)]{Connes73}.
Then \cite[Corollary~5.7]{AHHM} supplies an irreducible hyperfinite
type $\mathrm{III}_\lambda$ subfactor $P\subset M$ with
$\sigma_s^\varphi(P)=P$ for every $s\in\mathbb R$.
There is consequently a faithful normal conditional expectation
$E:M\to P$ with $\varphi\circ E=\varphi$, by
\cite[Theorem~III.4.7.7]{Blackadar}.

For a general type III von Neumann algebra $M$, take its central decomposition
\[
 M=\int_X^\oplus M_x\,d\mu(x),\qquad Z(M)=L_\infty(X,\mu),
\]
where $(X,\mu)$ is a standard probability space and almost every
$M_x$ is a type III factor with separable predual;
see \cite[III.1.6.3--III.1.6.4]{Blackadar}.
The standard disintegration and measurable-selection reduction used in
\cite[proof of Theorem~10.17]{Marrakchi25} applies here: the factor
results above give measurable choices of unital hyperfinite type III
subfactors $P_x\subset M_x$ and faithful normal conditional
expectations $E_x:M_x\to P_x$. Put
\[
 \begin{gathered}
 P:=\int_X^\oplus P_x\,d\mu(x),\qquad
 E:M\longrightarrow P,\\
 E(a):=\int_X^\oplus E_x(a_x)\,d\mu(x)
 \qquad\left(a=\int_X^\oplus a_x\,d\mu(x)\in M\right).
 \end{gathered}
\]
Then $E$ is a faithful normal conditional expectation.
The direct-integral criteria for type and injectivity show that $P$
is of type III and hyperfinite
\cite[III.1.6.4, Proposition~IV.2.2.12 and Corollary~IV.2.6.2]{Blackadar}.
Finally, the preadjoint $E_*:P_*\to M_*$ is an isometry, so $P_*$
is separable.
{Indeed, $E_*\omega=\omega\circ E$ for $\omega\in P_*$. Since
$E|_P=\mathrm{id}_P$ and $\|E\|=1$, we have
\[
 \|\omega\|=\|(\omega\circ E)|_P\|
 \le\|E_*\omega\|\le\|\omega\|.
\]}
\end{proof}

\subsection{Completion of the proof \revision{of Theorem~\ref{thm:unified}}}
\label{sec:nonfactor}

\begin{lemma}\label{lem:separable-support}
{Let $N$ be a semifinite von Neumann algebra, let $0<p<\infty$, and let
$E\subset L_p(N)$ be a nonzero separable subspace. There are a finite
von Neumann algebra $F$ with a faithful normal tracial state and a
linear isometry}
\[
 {I:E\longrightarrow L_p(F\overline\otimes B(\ell_2)).}
\]
\end{lemma}

\begin{proof}
Fix a faithful normal semifinite trace $\tau$ on $N$, and identify
$L_p(N)$ with the tracial space $L_p(N,\tau)$.
Choose a sequence $(y_k)$ dense in the unit ball of $E$ for the topology
induced by $\|\cdot\|_p$, and set
\[
 h=\sum_{k=1}^{\infty}2^{-k}
       \bigl(|y_k|^p+|y_k^*|^p\bigr)\in L_1(N,\tau)_+,\qquad
 e:=s(h),\qquad B:=eNe.
\]
The series converges in $L_1(N)$, and $e$ dominates the left and
right supports of every $y_k$. Hence,
\[
 E\subset eL_p(N,\tau)e=L_p(B,\tau|_B).
\]
Moreover, 
 $B$ is $\sigma$-finite and semifinite.
{Indeed, the left and right support projections of each
$y_k$ are $\sigma$-finite, and $e$ is their countable join.}

The structure theorem \cite[V.1.40]{Takesaki}, applied to the
$\sigma$-finite properly infinite semifinite algebra
$B\overline\otimes B(\ell_2)$, gives a finite $\sigma$-finite
von Neumann algebra $F$ and a normal $*$-isomorphism
\begin{equation}\label{eq:semifinite-stabilization}
 \Theta:B\overline\otimes B(\ell_2)
       \xrightarrow{\ \cong\ }F\overline\otimes B(\ell_2).
\end{equation}
The algebra $F$ has a faithful normal tracial state $\tau_F$
\cite[Corollary~III.2.5.8]{Blackadar}.
Put $\tau_1:=\tau|_B\otimes\Tr$ and $\tau_2:=\tau_F\otimes\Tr$.
By the Radon--Nikodym theorem for traces
\cite[Theorems~7.25 and~7.22(1)]{Hiai2021}, there is a positive operator
$d$ affiliated with the center of $F\overline\otimes B(\ell_2)$ such that
$\tau_1(a)=\tau_2(d\,\Theta(a))$ for every positive $a$.
For bounded $x\in L_p(B\overline\otimes B(\ell_2),\tau_1)$, put
$\Theta_p(x)=d^{1/p}\Theta(x)$. Since $d$ is central, we have
\[
 \|\Theta_p(x)\|_{L_p(\tau_2)}^p
 =\tau_2\bigl(d\,\Theta(|x|^p)\bigr)
 =\tau_1(|x|^p)=\|x\|_{L_p(\tau_1)}^p .
\]
Thus, $\Theta_p$ extends by density to a linear isometry for every $0<p<\infty$.
For a rank-one projection $e_{11}\in B(\ell_2)$, the required isometry is
\[
 I:E\longrightarrow L_p(F\overline\otimes B(\ell_2)),\qquad
 I(y):=\Theta_p(y\otimes e_{11})\quad(y\in E\subset L_p(B)).
\]
This completes the proof. 
\end{proof}

\begin{proof}[Proof of Theorem~\ref{thm:unified}]
{
\newtext{Let $e$ be the support of a normal state $\omega$ on $M$.
Then $eMe$ is a nonzero $\sigma$-finite type III von Neumann algebra.
By Halpern~\cite[Proposition~14, p.~318]{Halpern1978}, there is a faithful
normal conditional expectation from $eMe$ onto a countably generated
type III von Neumann subalgebra $M_0$. The GNS representation of $M_0$
associated with $\omega|_{M_0}$ is faithful and acts on a separable
Hilbert space, so $M_0$ has separable predual. By \eqref{eq:expected-isometry}
and \cite[Fact~1.4, p.~155]{RX}, we have isometric embeddings
\[
 L_p(M_0)\hookrightarrow L_p(eMe)=eL_p(M)e\hookrightarrow L_p(M).
\]
It therefore suffices to prove the theorem when $M$ has separable predual.}

Suppose that $T:L_p(M)\hookrightarrow L_p(N)$ for some $0<p<2$.
The Mazur map $u|x|\mapsto u|x|^{1/p}$ maps the unit ball of
$L_1(M)=M_*$ continuously onto the unit ball of $L_p(M)$
\cite[Lemma~3.2, p.~58]{Raynaud2002}. Hence, $L_p(M)$ is separable.
Applying Lemma~\ref{lem:separable-support} to the separable subspace
$T(L_p(M))\subset L_p(N)$ gives a finite {von Neumann} algebra $F$ with a faithful
normal tracial state and an isometry
\[
 I:T(L_p(M))\longrightarrow L_p(F\overline\otimes B(\ell_2)).
\]
Consequently,
\begin{equation}\label{eq:general-target-embedding}
 \widetilde T:=IT:L_p(M)\hookrightarrow
                   L_p(F\overline\otimes B(\ell_2)).
\end{equation}
For $p=1$, the embedding \eqref{eq:general-target-embedding}
contradicts Theorem~\ref{thm:main}.
For $p\in(0,1)\cup(1,2)$, Proposition~\ref{prop:expected-typeiii-general}
gives a unital hyperfinite type III subalgebra $P\subset M$ with
separable predual and a faithful normal conditional expectation
$E:M\to P$. The isometry $j_p:L_p(P)\to L_p(M)$ from
\eqref{eq:expected-isometry} gives the isomorphic embedding
\[
 L_p(P)\xrightarrow{\ j_p\ }L_p(M)
       \xrightarrow{\ \widetilde T\ }
       L_p(F\overline\otimes B(\ell_2)),
\]
contrary to Theorem~\ref{thm:interior}.
}
\end{proof}

\begin{corollary}\label{cor:mixed-types}
Let $A$ be a von Neumann algebra and let $N$ be semifinite.
If $L_p(A)\hookrightarrow L_p(N)$ for some
$0<p<2$, then $A$ is semifinite. The same conclusion holds
if $L_p(A)$ is isomorphic to a complemented subspace of $L_p(N)$
for some $2<p<\infty$.
\end{corollary}
\begin{proof}
{ 
Suppose that $A$ is not semifinite. Its type decomposition {implies the existence of}
$0\ne z\in\operatorname{Proj}(Z(A))$ such that $zA$ has type III
\cite{Takesaki}.
For an isomorphic embedding $T:L_p(A)\hookrightarrow L_p(N)$, set
\[
 T_z:=T|_{L_p(zA)}:L_p(zA)\hookrightarrow L_p(N),
 \qquad L_p(zA)=zL_p(A)z\mbox{\cite[Fact~1.4, p.~155]{RX}}.
\]
When $0<p<2$, the embedding $T_z$ contradicts
Theorem~\ref{thm:unified}.
{When $2<p<\infty$ and $L_p(A)$ is isomorphic to a complemented subspace of $L_p(N)$,
so is $L_p(zA)$ (since $L_p(zA)$ is a complemented subspace of $L_p(A)$
by $L_p(zA)=zL_p(A)z$), contradicting Corollary~\ref{cor:duality}.}
}
\end{proof}

\begin{remark}
As shown in the Introduction, for every nonzero hyperfinite type III
von Neumann algebra $M$ with separable predual and every
$1\le p<\infty$, $L_p(M)$ embeds isometrically as a
$1$-complemented subspace in the space $L_p(R_\infty)$.
On the other hand,
  $L_p$-spaces over nonhyperfinite type III algebras with separable
preduals may not be isomorphic to a complemented subspace of $L_p(R_\infty)$. Indeed, Junge, Nielsen, Ruan and
Xu~\cite[Theorem~2.19, p.~284]{JNRX2004} proved the existence of a
finite von Neumann algebra $F$ with separable predual such that
$L_p(F)$ ($80<p<\infty$) fails the approximation property. 
Since the approximation property passes to
complemented subspaces, $L_p(F\overline\otimes R_\infty)$ fails this property.
However, $L_p(R_\infty)$ has the metric approximation property for $1<p<\infty$
\cite[Theorem~5.9, p.~311]{JNRX2004}.
\end{remark}

\section*{AI statement}
The authors thank Ran Xu for discussions concerning the case $0<p<1$.
GPT 6.0 Astra assisted with calculations in the original manuscript and
with locating references. AI tools also assisted in drafting the
unified proof for the cases when $p>1$ and $p\le 1$ in Section~\ref{sec:array-proof}.

\end{document}